\documentclass[10pt, reqno]{amsart}

\usepackage[a4paper,margin=1.1in]{geometry}

\usepackage{amsmath,amssymb,amsthm, epsfig}
\usepackage{mathalfa}
\usepackage{amsfonts}
\usepackage{amsmath}
\usepackage{amssymb}
\usepackage[mathscr]{euscript}
\usepackage{esint}
\usepackage{hyperref}
\usepackage[nameinlink]{cleveref}
\hypersetup{colorlinks=true,linkcolor=blue,citecolor=blue,urlcolor=blue}
\usepackage{enumitem}
\usepackage[dvipsnames]{xcolor}
\usepackage{comment}

\makeatletter
\renewcommand{\subsection}{\@startsection{subsection}{2}%
	\z@{\linespacing\@plus.7\linespacing}{.5\linespacing}%
	{\normalfont\scshape}}
\makeatother

\newtheorem{theorem}{Theorem}[section]

\newtheorem{definition}{Definition}[section]

\newtheorem{lemma}{Lemma}[section]
\newtheorem{corollary}{Corollary}[section]
\newtheorem{proposition}{Proposition}[section]
\newtheorem{remark}{Remark}[section]

\newcommand{\intav}[1]{\mathchoice {\mathop{\vrule width 6pt height 3 pt depth  -2.5pt
			\kern -8pt \intop}\nolimits_{\kern -6pt#1}} {\mathop{\vrule width
			5pt height 3  pt depth -2.6pt \kern -6pt \intop}\nolimits_{#1}}
	{\mathop{\vrule width 5pt height 3 pt depth -2.6pt \kern -6pt
			\intop}\nolimits_{#1}} {\mathop{\vrule width 5pt height 3 pt depth
			-2.6pt \kern -6pt \intop}\nolimits_{#1}}}

\title[Gradient regularity with variable degeneracy]{Gradient Regularity for Fully Nonlinear Equations with Variable Degeneracy and Hamiltonian Lower-Order Terms}

\author[G. Cosmo, R. R. Costa, \and D. Marcon]{Gleiciano Cosmo, Rafael R. Costa, \and Diego Marcon}

\thanks{%
	\begin{tabular}{@{}l}
		Instituto de Matemática e Estatística, Universidade Federal do Rio Grande do Sul, Porto Alegre, Brazil. \\
		emails: \url{gleiciano.santos@ufrgs.br}, \url{ramos.costa@ufrgs.br}, \url{diego.marcon@ufrgs.br}.
\end{tabular}}

\begin{document}

	\begin{abstract} 
		We study local regularity properties of viscosity solutions to fully nonlinear elliptic equations with variable gradient degeneracy and Hamiltonian-type lower-order terms,
		\[
		|\nabla u|^{p(x)}F(\nabla^{2}u)
		+
		a(x)|\nabla u|^{q(x)}
		=
		f(x).
		\]
		Here, $F$ is uniformly elliptic, while the exponents $p$ and $q$ are allowed to vary in space. We prove interior Hölder estimates for the gradient, with an exponent determined by the maximal degeneracy rate and by the regularity available for the associated homogeneous uniformly elliptic equation. We also obtain pointwise improvements at points where the source term and the Hamiltonian coefficient vanish with prescribed Hölder rates. Finally, at extremal points, we establish a Schauder-type estimate showing that the solution separates from its extremal value with order strictly larger than two. The proofs combine compactness estimates for shifted equations, stability of viscosity solutions, and improvement-of-flatness iterations.
		
		\bigskip
		
		\noindent{\it 2020 Mathematics Subject Classification:} 35B65, 35J60, 35J70, 35D40.
		
		\noindent\textit{Keywords:} {Fully nonlinear elliptic equations; variable degeneracy; Hamiltonian terms; viscosity solutions; H\"older gradient estimates.}
	\end{abstract}
	\vspace{3mm}
	%{\small \tableofcontents}
	\date{\today}
	
	\maketitle
	
	\section{Introduction} \label{sct intro}
	
	In this paper, we study local regularity properties of viscosity solutions to fully nonlinear elliptic equations with variable degeneracy and Hamiltonian-type lower-order terms. More precisely, we consider equations of the form
	\begin{equation}\label{eq principal}
		|\nabla u|^{p(x)}F(\nabla^{2}u)+a(x)|\nabla u|^{q(x)}=f(x)
		\quad \text{in } \Omega,
	\end{equation}
	where $\Omega\subset\mathbb{R}^{n}$ is a bounded domain, $F$ is a uniformly elliptic fully nonlinear operator, $a$ is a bounded coefficient, and $p,q$ are variable exponents satisfying suitable structural assumptions.
	
	The first term in \eqref{eq principal} represents a fully nonlinear diffusion whose ellipticity degenerates along the critical set $\{\nabla u=0\}$: the factor $|\nabla u|^{p(x)}$ modulates the strength of the elliptic operator according to both the size of the gradient and the spatial position, with $p(x)$ measuring the local rate of degeneracy. The term $a(x)|\nabla u|^{q(x)}$ is a Hamiltonian-type lower-order contribution, depending only on the first-order behavior of $u$; here, $a(x)$ measures its local intensity, while $q(x)$ determines its growth in the gradient variable. Thus, the relation between $p(x)$ and $q(x)$ is central to the scaling of the equation and to the control of the Hamiltonian term. Finally, $f(x)$ represents a forcing term, whose interaction with the degeneracy can be nontrivial.

	The main contribution of this paper is to show that this combination of effects still preserves a robust regularity theory. First, in \Cref{teo_c1alpha}, we establish interior Hölder estimates for the gradient of viscosity solutions to \eqref{eq principal}, with an exponent determined by the worst degeneracy rate of the leading term and by the regularity available for the associated homogeneous uniformly elliptic equation. We then prove in \Cref{Thm2} that this regularity can be improved pointwise at locations where the source term and the Hamiltonian coefficient vanish with prescribed Hölder rates. Finally, in \Cref{SchauderEstimate},  at extremal points, we obtain a Schauder-type improvement at the level of second-order growth: although no full Hessian expansion is asserted, the solution separates from its extremal value with order strictly larger than two. In this sense, our results extend the known regularity theory for gradient-degenerate fully nonlinear equations to a variable-exponent setting with an additional first-order Hamiltonian term.

	A basic model underlying equation \eqref{eq principal} is the degenerate fully nonlinear equation
	\begin{equation}\label{eq:basic-model}
		|\nabla u|^{\gamma}F(\nabla^{2}u)=f
		\quad \text{in } B_1,
	\end{equation}
	where $\gamma\geq 0$ is constant and $F$ is uniformly elliptic. Equations of this type, and their singular counterparts, are developed in a series of works by Birindelli and Demengel; see, for instance, \cite{BD04,BD07,BD09,BD10}. In particular, gradient regularity is first obtained in special settings, such as radial solutions, and later for Dirichlet problems associated with fully nonlinear degenerate elliptic equations \cite{BD12,BD14}.
	
	A general interior $\mathrm{C}^{1,\alpha}$ theory for \eqref{eq:basic-model} is established by Imbert and Silvestre \cite{IS13}. Their proof combines compactness estimates, the Jensen--Ishii maximum principle, and an improvement-of-flatness iteration that approximates the graph of $u$ by affine functions on smaller and smaller balls. This approach shows that the regularity of \eqref{eq:basic-model} is governed by two different obstructions: the degeneracy exponent $\gamma$ and the regularity available for the homogeneous uniformly elliptic equation.
	More precisely, the Hölder exponent for the gradient cannot, in general, exceed $1/(1+\gamma)$. This can already be seen from explicit radial examples, such as
	\[
	u(x)=c|x|^{1+\frac{1}{1+\gamma}},
	\]
	which gives a viscosity solution of \eqref{eq:basic-model}, for suitable choices of $c$, when $F(\nabla^2u)=\operatorname{tr}(\nabla^2u)=\Delta u$ and $f$ is constant. Its gradient behaves like $|x|^{1/(1+\gamma)}$ near the origin. On the other hand, even in the uniformly elliptic homogeneous regime, we are limited by the exponent $\alpha_0\in(0,1)$ for which solutions of
	\[
	F(\nabla^2 h)=0
	\quad \text{in } B_1
	\]
	are locally of class $\mathrm{C}^{1,\alpha_0}$; see, for instance, \cite{Caffa-Cabre}. For background on universal continuity estimates for fully nonlinear elliptic equations, see also \cite{EdTeix}. Hence, in the general uniformly elliptic case, we obtain that, locally, solutions of \eqref{eq:basic-model} have Hölder continuous gradients with exponent
	\[
	\alpha\in(0,\alpha_0)\cap
	\left(0,\frac{1}{1+\gamma}\right],
	\]
	with estimate
	\[
	\|h\|_{\mathrm{C}^{1,\alpha_0}(B_{1/2})}
	\leq C\|h\|_{\mathrm{L}^{\infty}(B_1)}.
	\]
	When $F$ is convex or concave, the optimal exponent dictated by the degeneracy is achieved, as in \cite[Corollary 3.2]{ART15}. Geometric compactness methods for nonvariational singular elliptic equations are also developed in \cite{AraujoRicarteTeixeira17}, providing an important perspective on the regularity theory for equations whose ellipticity degenerates or becomes singular near critical points. Sharp regularity phenomena related to critical exponents also appear in the study of degenerate elliptic equations, notably in connection with the $C^{p'}$-regularity conjecture; see \cite{AraujoTeixeiraUrbano17,AraujoTeixeiraUrbano18}.

	Several developments of \eqref{eq:basic-model} have since been obtained in different directions. For variable-exponent degeneracies, Bronzi, Pimentel, Rampasso, and Teixeira \cite{BPRT20} study equations of the form
	\[
	|\nabla u|^{\theta(x)}F(\nabla^2u)=f(x),
	\]
	under the condition that $\inf\theta>-1$, and prove local $\mathrm{C}^{1,\alpha}$ estimates with constants that are independent of the modulus of continuity of the exponent. In the same direction, Fang, Rădulescu, and Zhang \cite{FRZ21} consider variable-exponent double-phase type degeneracies of the form
	\[
	\left[|\nabla u|^{p(x)}+a(x)|\nabla u|^{q(x)}\right]F(\nabla^2u)=f(x),
	\]
	obtaining sharp local gradient Hölder estimates by means of compactness, scaling, and refined im-provement-of-flatness arguments.
	
	The interest in variable exponents and related nonstandard growth structures is also motivated by variational problems, where the growth law may encode spatially dependent features of the medium or of the data; see, for instance, \cite{ChallalLyaghfouriRodriguesTeymurazyan14,MarconRodriguesTeymurazyan18,RodriguesTeymurazyan11}. This flexibility is particularly useful, for instance, in image restoration and related image-processing models: in \cite{ChenLevineRao06}, variable-exponent functionals are used to combine total-variation type regularization near edges with more diffusive smoothing effects in homogeneous regions; see also \cite{LiLiPi10} for related variable-exponent models in image restoration. Although the present paper is set in the fully nonlinear nonvariational framework, this variational background provides part of the motivation for studying elliptic equations whose degeneracy varies from point to point.
	
	A related line of research concerns non-homogeneous degeneracies and geometric tangential methods. Regularity estimates for fully nonlinear elliptic equations with nonhomogeneous degeneracy are obtained in \cite{DeFilippis21}. In \cite{DaSilvaRicarte20}, da Silva and Ricarte obtain sharp gradient estimates for fully nonlinear models whose prototype is
	\[
	\Big[|\nabla u|^{p}+a(x)|\nabla u|^{q}\Big]\mathcal{M}^{+}_{\lambda,\Lambda}(\nabla^2u)=f(x,u).
	\]
	Further global estimates for equations with unbalanced variable degeneracy are established in \cite{BezerraDaSilvaRampassoRicarte23}. Nonlocal analogues of gradient-degenerate elliptic equations are studied, for instance, in \cite{AndradeDosPrazeresSantos24,DT21}, where interior regularity estimates are obtained for fractional or integro-differential equations that degenerate with the gradient.
	
	The presence of Hamiltonian lower-order terms introduces an additional difficulty, since the first-order contribution may compete with the elliptic diffusion at different scales. More recently, for constant exponents, optimal quantitative gradient estimates for degenerate elliptic equations with Hamiltonian terms are investigated in \cite{AndradeNascimento25}. Related regularity questions for fully nonlinear equations with gradient dependence or Hamiltonian-type growth have been studied in several settings; see, for instance, \cite{Birin-Demen,Nornberg19}. Schauder-type estimates at special points for fully nonlinear degenerate elliptic equations were also obtained in \cite{TNasc,EdT}. 
	%
%	The presence of Hamiltonian lower-order terms introduces an additional difficulty, since the first-order contribution may compete with the degenerate elliptic diffusion at different scales. For constant exponents, Hölder regularity of the gradient for fully nonlinear equations with Hamiltonian terms is obtained in \cite{Birin-Demen}. More recently, optimal quantitative gradient estimates for degenerate elliptic equations with Hamiltonian terms are investigated in \cite{AndradeNascimento25}. Schauder-type estimates at special points for fully nonlinear degenerate elliptic equations are also obtained in \cite{TNasc}. 
	%
	We also mention related developments on sharp regularity and geometric properties of free boundaries in degenerate obstacle, dead-core, quenching, and $\infty$-Laplacian type problems; see, for instance, \cite{AmaralDaSilvaRicarteTeymurazyan19,AraujoLeitaoTeixeira16,AraujoRicarteTeixeira17,AraujoTeymurazyanUrbano24,AraujoTeymurazyanVoskanyan23,DiehlTeymurazyan20,DosPrazeresTeymurazyanUrbano25,RicarteDaSilvaTeymurazyan17}.

	The present work combines these directions by allowing both variable degeneracy exponents and a Hamiltonian lower-order term in \eqref{eq principal}, and by deriving not only local gradient estimates but also pointwise improvements under vanishing assumptions on the data.

%	-----------------------------
%	
%	The purpose of the present paper is to combine these perspectives by studying the local gradient regularity of viscosity solutions to equations with variable degeneracy and Hamiltonian-type lower-order terms, namely \eqref{eq principal}. \textcolor{red}{Comentar os outros resultados}
%	
	
	\smallskip

The paper is organized as follows.  In \Cref{hipoteses}, we fix the notation, the structural assumptions, the viscosity framework, and state our three main results. In \Cref{regularidade secao}, we prove Hölder compactness estimates for shifted equations by means of the doubling of variables method. These estimates are used in \Cref{sec_compactness} to establish the approximation lemma, which allows normalized solutions to be compared with solutions of homogeneous uniformly elliptic equations. In \Cref{sec_c1alpha}, we prove the local Hölder regularity of the gradient by an improvement-of-flatness iteration. The pointwise improvement under Hölder vanishing assumptions on the data is proved in \Cref{sec:improved-pointwise}. Finally, in \Cref{sec:schauder-extremal}, we prove the pointwise Schauder-type estimate at extremal points.

	\section{Notation and Main Results}\label{hipoteses}
	
	In this section, we fix the notation, structural assumptions, and viscosity framework used throughout the paper. We also state our main regularity results. The first theorem gives local Hölder continuity of the gradient for the full equation. The second and third results concern pointwise improvements at those points where the source term and the Hamiltonian coefficient vanish with prescribed Hölder rates.

	\subsection{Assumptions, Definitions, and Preliminaries}
	
	We begin by recalling the ellipticity framework. Let $\mathcal S(n)$ denote the space of real symmetric $n\times n$ matrices. Throughout the paper, $0<\lambda\leq\Lambda$ are fixed ellipticity constants.
	
	\smallskip
	
	\begin{enumerate}[label=$(\mathrm{A}\arabic*)$]
		\item \label{hipoteseF} For $M\in\mathcal S(n)$, denote by
		$\{e_i\}_{i=1}^n$ the eigenvalues of $M$. The Pucci extremal operators are defined by
		\[
		P^+_{\lambda,\Lambda}(M)
		:=
		\Lambda\sum_{e_i>0}e_i
		+
		\lambda\sum_{e_i<0}e_i
		\quad
		\text{and}
		\quad
		P^-_{\lambda,\Lambda}(M)
		:=
		\lambda\sum_{e_i>0}e_i
		+
		\Lambda\sum_{e_i<0}e_i.
		\]
		We assume that the operator $F:\mathcal S(n)\to\mathbb R$ is uniformly
		$(\lambda,\Lambda)$-elliptic, namely, that for every $M,N\in\mathcal S(n)$,
		\begin{equation}\label{eq:full-pucci-ellipticity}
			P^-_{\lambda,\Lambda}(M-N)
			\le
			F(M)-F(N)
			\le
			P^+_{\lambda,\Lambda}(M-N).
		\end{equation}
		In particular, if $N\ge0$, then \eqref{eq:full-pucci-ellipticity} gives
		\[
		\lambda\operatorname{tr}(N)
		\le
		F(M+N)-F(M)
		\le
		\Lambda\operatorname{tr}(N).
		\] 
		Conversely, this monotonicity condition for every $N\ge0$ is equivalent to
		\eqref{eq:full-pucci-ellipticity}. We also assume, without
		loss of generality, that
		$F(0)=0$.
		
		\vspace{0.5em}
		
		\item \label{hip2} (Structural conditions) We assume that $a, p, q , f \in \mathrm{C}^{0}(\Omega)\cap L^{\infty}(\Omega)$ and, for all $x\in \Omega$,
		\[
		0\leq q(x)\leq p(x) \le p^{+} := \sup_{\Omega} p < + \infty.
		\]
	\end{enumerate}
	
	With these structural assumptions, we next specify the notion of solution of \eqref{eq principal} used in the sequel. The appropriate definition in our setting is that of viscosity solutions.

	\begin{definition}[Viscosity solutions]\label{def:viscosity}
		A function $u\in \mathrm{C}^{0}(\Omega)$ is a viscosity supersolution of
		\eqref{eq principal} if, whenever $\varphi\in C^2(\Omega)$ and
		$u-\varphi$ has a local minimum at $x_0\in\Omega$, we have
		\[
		|\nabla\varphi(x_0)|^{p(x_0)}F(\nabla^2\varphi(x_0))
		+
		a(x_0)|\nabla\varphi(x_0)|^{q(x_0)}
		\le
		f(x_0).
		\]
		It is a viscosity subsolution if, whenever $\varphi\in C^2(\Omega)$ and
		$u-\varphi$ has a local maximum at $x_0\in\Omega$,  we have
		\[
		|\nabla\varphi(x_0)|^{p(x_0)}F(\nabla^2\varphi(x_0))
		+
		a(x_0)|\nabla\varphi(x_0)|^{q(x_0)}
		\ge
		f(x_0).
		\]
		A viscosity solution is a function that is both a viscosity supersolution and a
		viscosity subsolution.
	\end{definition}
	
	We also use the same definition for shifted equations, in which the gradient variable is replaced by $\nabla u+\xi$, with $\xi\in\mathbb R^n$ fixed. This notation is convenient in the compactness and improvement-of-flatness arguments of \Cref{regularidade secao}.

	\smallskip
	
	The main technical device used to obtain the initial compactness estimates is the doubling of variables method. For convenience, we recall the standard Jensen--Ishii maximum principle; for more details and its proof, see \cite[Theorem 3.2]{Crandall1992user}.

	\begin{proposition}[Jensen--Ishii maximum principle]\label{prop:jensen-ishii}
		Let $\Omega\subset\mathbb R^n$ be open and let
		$\Phi\in C^2(\Omega\times\Omega)$. Let $u\in \mathrm{C}^{0}(\Omega)$ be a viscosity solution of
		\[
		G(x,\nabla u,\nabla^2u)=0,
		\quad\text{in }\Omega,
		\]
		where $G:\Omega\times\mathbb R^n\times\mathcal S(n)\to\mathbb R$ is continuous.
		Suppose that $(\bar x,\bar y)\in\Omega\times\Omega$ is a local maximum point of
		\[
		u(x)-u(y)-\Phi(x,y).
		\]
		Then, for every $\varepsilon>0$, there exist matrices
		$X,Y\in\mathcal S(n)$ such that the equation may be tested from above at
		$\bar x$ with first-order part $\nabla_x\Phi(\bar x,\bar y)$ and from below at $\bar y$ with first-order part $-\nabla_y\Phi(\bar x,\bar y)$. In particular, with the sign convention of \Cref{def:viscosity}, we have
		\[
		G\bigl(\bar x,\nabla_x\Phi(\bar x,\bar y),X\bigr)\ge0
		\]
		and
		\[
		G\bigl(\bar y,-\nabla_y\Phi(\bar x,\bar y),Y\bigr)\le0.
		\]
		Moreover, for $A:=\nabla^2\Phi(\bar x,\bar y)$, there holds
		\begin{equation}\label{eq:jensen-ishii-matrix}
			-\left(\frac1\varepsilon+\|A\|\right)I
			\le
			\begin{pmatrix}
				X & 0\\
				0 & -Y
			\end{pmatrix}
			\le
			A+\varepsilon A^2.
		\end{equation}
	\end{proposition}
	
	In the applications
	below, the operator $G$ has the form
	\[
	G(x,\zeta,M)
	:=
	|\zeta|^{p(x)}F(M)
	+
	a(x)|\zeta|^{q(x)}
	-
	f(x),
	\]
	or the corresponding shifted version
	\[
	G_\xi(x,\zeta,M)
	:=
	|\zeta+\xi|^{p(x)}F(M)
	+
	a(x)|\zeta+\xi|^{q(x)}
	-
	f(x).
	\]
	
	\medskip
	
	The limiting equation in our compactness arguments is a uniformly elliptic homogeneous equation. We  recall the standard interior $\mathrm{C}^{1,\alpha_0}$ estimate that serves as the regularity input in the approximation procedure. Namely, see \cite{Caffa-Cabre}, there exist universal constants
	$\alpha_0\in(0,1)$ and $C>0$, depending only on $n,\lambda,\Lambda$, such that: if
	$h\in \mathrm{C}^{0}(B_1)$ is a viscosity solution of
	\[
	F(\nabla^2h)=0,
	\quad\text{in }B_1;
	\]
	then,
	\[
	\|h\|_{\mathrm{C}^{1,\alpha_0}(B_{1/2})}
	\le
	C\|h\|_{\mathrm{L}^{\infty}(B_1)}.
	\]

	\subsection{Main Results}
	
	We now state the main results of our paper. The first theorem gives the local Hölder regularity of the gradient for solutions of \eqref{eq principal}. The admissible exponent is constrained by two mechanisms: the degeneracy rate of the equation, measured by $p^+$, and the regularity exponent $\alpha_0$ available for the homogeneous uniformly elliptic equation.

	\begin{theorem}[H\"older regularity of the gradient]\label{teo_c1alpha}
		Assume \ref{hipoteseF} -- \ref{hip2}, and let $u\in \mathrm{C}^{0}(\Omega)$ be a viscosity solution of \eqref{eq principal}. Then, $u\in \mathrm{C}^{1,\alpha}_{\rm loc}(\Omega)$ for any $\alpha$ in the range
		\begin{equation}\label{eqn:alpha-range}
			\alpha\in(0,\alpha_0)\cap\left(0,\frac{1}{1+p^{+}}\right].
		\end{equation}
		Moreover, we have the following H\"{o}lder estimate for the gradient of $u$: for fixed $\Omega'\subset\subset\Omega$, there exists a constant $c>0$, depending on $n,\lambda,\Lambda,\alpha,\operatorname{dist}(\Omega',\partial\Omega)$, and $p^{+}$, such that
		$$
		[\nabla u]_{\mathrm{C}^{0,\alpha}(\Omega')} \leq c \left(1+\|u\|_{\mathrm{L}^{\infty}(\Omega)}+\|a\|^{\frac{1}{1+\inf_{\Omega}(p-q)}}_{\mathrm{L}^{\infty}(\Omega)} +\|f\|^\frac{1}{1+\inf_\Omega p}_{\mathrm{L}^{\infty}(\Omega)}\right).
		$$
	\end{theorem}

	\begin{remark}
		When $F$ is convex or concave, the obstruction coming from the homogeneous equation is improved by the so called Evans--Krylov theory, and the optimal exponent dictated by the degeneracy can be recovered; see \cite[Corollary 3.2]{ART15}. Thus, in that structural setting, \Cref{teo_c1alpha} extends the sharp exponent previously obtained for constant degeneracy to the present variable-exponent framework, namely $1/(1+p^{+})$.
	\end{remark}

	The exponent in \Cref{teo_c1alpha} is limited by the worst degeneracy rate $p^+$. However, this obstruction may be weakened at points where the lower-order data vanish with additional Hölder decay. Motivated by the pointwise improvement obtained in \cite{TNasc} for degenerate fully nonlinear equations, we next prove a higher pointwise regularity estimate for \eqref{eq principal} under vanishing assumptions on the source term and on the Hamiltonian coefficient.
	More precisely, we assume that, for some constants $K_1,K_2>0$ and $\theta_1,\theta_2\in(0,1)$,
	\begin{equation}\label{source1}
		|f(x)|\le K_1|x|^{\theta_1}
		\quad\text{and}\quad
		|a(x)|\le K_2|x|^{\theta_2}
		\quad\text{in }B_1.
	\end{equation} 
	
	Under this additional information, the scaling of the equation allows an improvement of the pointwise expansion of $u$ at the origin. For convenience, denote
	\begin{equation}\label{eqn:d+}
		d^+:=\sup_{B_1} \, (p-q).
	\end{equation}

	\begin{theorem}\label{Thm2}
		Assume \ref{hipoteseF} -- \ref{hip2}, and let $u\in \mathrm{C}^{0}(B_1)$ be a viscosity solution of \eqref{eq principal}. Assume further that $f$ and $a$ satisfy \eqref{source1}. Then, $u$ is $\mathrm{C}^{1,\gamma}$ at the origin for any $\gamma$ in the range
		\begin{equation}\label{eqn:gamma-thm2}
			\gamma\in(0,\alpha_0)
			\cap\left(0,\frac{1+\theta_1}{1+p^+}\right]
			\cap\left(0,\frac{1+\theta_2}{1+d^+}\right],
		\end{equation} where $d^{+}$ is as in \eqref{eqn:d+}. Moreover, for every $0<r<1/2$,
		$$\sup_{B_r}|u(x)-u(0)-Du(0)\cdot x|\leq C r^{1+\gamma}$$
		where $C$ depends only upon $\|u\|_{L^{\infty}(B_1)}$ and universal parameters.
	\end{theorem}
	
	At extremal points, the previous pointwise analysis can be pushed one order further. Indeed, if the origin is a local minimum of $u$, then the affine part of the expansion reduces to the constant $u(0)$. Under stronger compatibility conditions between the H\"older decay of $f$ and $a$ and the degeneracy exponents, we obtain a pointwise Schauder-type estimate; this result is motivated by \cite{TNasc}.
	
	Here, we assume additionally that
	\begin{equation}\label{eq:theta-schauder-assumptions}
		0\le p^{+}:=\sup_{B_1}p <\theta_1<1,
		\qquad
		0\le d^+:=\sup_{B_1} \, (p-q)<\theta_2<1.
	\end{equation}
	Under these restrictions, denote
	\begin{equation}\label{eq:alpha-tilde-def}
		\tilde\alpha
		:=
		\min\left\{
		\frac{\theta_1-p^{+}}{1+p^{+}},
		\frac{\theta_2-d^{+}}{1+d^{+}}
		\right\}>0.
	\end{equation}
	
	Our third result is a pointwise Schauder-type estimate: although no full second-order expansion is asserted, the solution separates from its extremal value with order strictly larger than two.

	\begin{theorem}[Pointwise Schauder estimate at an extremal point]
		\label{SchauderEstimate}
		Assume \ref{hipoteseF} -- \ref{hip2}, and let $u\in \mathrm{C}^{0}(B_1)$ be a viscosity solution of \eqref{eq principal}. Assume in addition that
		\eqref{source1} and \eqref{eq:theta-schauder-assumptions} hold, and that the origin is a
		local minimum of $u$. Then, there exist constants $C>0$ and $r_0\in(0,1)$ such
		that
		\[
		\sup_{B_r}|u(x)-u(0)|
		\le
		C r^{2+\tilde\alpha}
		\]
		for every $0<r\le r_0$, where $\tilde\alpha$ is give by \eqref{eq:alpha-tilde-def}. The constant $C$ depends on
		$\|u\|_{\mathrm{L}^{\infty}(B_1)}$ and on the universal parameters $n,\lambda,\Lambda,p^{+},d^{+},\theta_1,\theta_2,K_1,K_2$.
	\end{theorem}

	\subsection{Scaling and smallness regime}\label{subsec_small}
	
	The following scaling is used repeatedly in the proofs we present. Its role is to reduce the estimates to normalized solutions with bounded oscillation and small lower-order data, without changing the ellipticity constants of the operator.
	
	Let $u$ be a viscosity
	solution of \eqref{eq principal} in $\Omega$. Fix $x_0\in\Omega$ and
	$0<\rho<\operatorname{dist}(x_0,\partial\Omega)$. For $K>0$, define
	\[
	v(x):=\frac{u(x_0+\rho x)}{K},
	\quad x\in B_1.
	\]
	Then
	\[
	\nabla v(x)=\frac{\rho}{K}\nabla u(x_0+\rho x)
	\quad\text{and}\quad
	\nabla^2v(x)=\frac{\rho^2}{K}\nabla^2u(x_0+\rho x).
	\]
	Consequently, $v$ solves an equation of the same type,
	\begin{equation}\label{eq:scaled-equation-prelim}
		|\nabla v|^{\bar p(x)}\bar F(\nabla^2v)
		+
		\bar a(x)|\nabla v|^{\bar q(x)}
		=
		\bar f(x),
		\quad\text{in }B_1,
	\end{equation}
	where
	\[
	\bar p(x):=p(x_0+\rho x),
	\qquad
	\bar q(x):=q(x_0+\rho x),\\[0.5em]
	\]
	\[
	\bar F(M):=
	\frac{\rho^2}{K}
	F\left(\frac{K}{\rho^2}M\right),
	\quad
	\bar f(x):=
	\frac{\rho^{\bar p(x)+2}}
	{K^{\bar p(x)+1}}
	f(x_0+\rho x),
	\]
	and
	\[
	\bar a(x):=
	\rho^{\bar p(x)-\bar q(x)+2}
	K^{\bar q(x)-\bar p(x)-1}
	a(x_0+\rho x).
	\]
	The operator $\bar F$ is still uniformly $(\lambda,\Lambda)$-elliptic and
	$\bar F(0)=0$.
	This scaling allows us to work in a normalized smallness
	regime: indeed, choosing $K\ge1$ large enough in terms of
	\[
	1+\|u\|_{\mathrm{L}^{\infty}(\Omega)}
	+
	\|a\|_{\mathrm{L}^{\infty}(\Omega)}^{\frac{1}{1+\inf_\Omega(p-q)}}
	+
	\|f\|_{\mathrm{L}^{\infty}(\Omega)}^{\frac{1}{1+\inf_\Omega p}},
	\]
	and then choosing $\rho\in(0,1)$ sufficiently small in terms of the desired
	smallness parameter, we may assume
	\begin{equation}\label{eq_smallnessregime}
		\|v\|_{\mathrm{L}^{\infty}(B_1)}\le1
		\quad\text{and}\quad
		\|\bar a\|_{\mathrm{L}^{\infty}(B_1)}
		+
		\|\bar f\|_{\mathrm{L}^{\infty}(B_1)}
		\le\varepsilon.
	\end{equation}
	Estimates proved for $v$ can then be transferred back to $u$, with constants
	depending explicitly on $K$ and $\rho$.
	
	We also use this scaling under pointwise growth assumptions. If, at the
	base point $x_0$, we have
	\[
	|f(x)|\le K_1|x-x_0|^{\theta_1},
	\quad\text{and}\quad
	|a(x)|\le K_2|x-x_0|^{\theta_2},
	\]
	then the rescaled coefficients satisfy
	\[
	|\bar f(x)|
	\le
	K_1
	\frac{\rho^{\bar p(x)+2+\theta_1}}
	{K^{\bar p(x)+1}}
	|x|^{\theta_1},
	\]
	and
	\[
	|\bar a(x)|
	\le
	K_2
	\rho^{\bar p(x)-\bar q(x)+2+\theta_2}
	K^{\bar q(x)-\bar p(x)-1}
	|x|^{\theta_2}.
	\]
	Thus, the same normalization procedure can be combined with the growth
	assumptions used in the pointwise estimates of the final sections.

	\section{H\"older compactness estimates}\label{regularidade secao}
	
	In this section, we prove estimates that are used in the
	approximation lemma of \Cref{sec_compactness}. The point is that, after subtracting affine functions in the
	improvement-of-flatness iteration, the equation naturally acquires a shifted
	gradient. Thus, for fixed $\xi\in\mathbb R^n$, we need to consider viscosity solutions of
	\begin{equation}\label{eq:shifted}
		|\nabla u+\xi|^{p(x)}F(\nabla^{2}u)
		+
		a(x)|\nabla u+\xi|^{q(x)}
		=
		f(x)
		\quad\text{in } B_1.
	\end{equation}
	
	Throughout this section, we assume
	%
%	\begin{comment}
%		we assume that $F:\mathcal S(n)\to\mathbb R$ is
%		uniformly $(\lambda,\Lambda)$-elliptic, $F(0)=0$, and that
%		\[
%		0\le q(x)\le p(x)\le p^{+}
%		\qquad\text{in }B_1.
%		\]
%		We also assume
%	\end{comment} 
	that $u$ is normalized, namely, that
	\[
	\|u\|_{\mathrm{L}^{\infty}(B_1)}\le 1.
	\]
	The constants below are allowed to depend on $n,\lambda,\Lambda,p^{+}$, and on the
	radius of the compact subdomain. If we are working under the smallness regime
	\[
	\|a\|_{\mathrm{L}^{\infty}(B_1)}+\|f\|_{\mathrm{L}^{\infty}(B_1)}\le 1;
	\]
	then, the constants are universal in the usual sense. Otherwise, their dependence
	on $\|a\|_\infty$ and $\|f\|_\infty$ are explicitly indicated.
	
	The proof is based on the standard doubling of variables argument of \Cref{prop:jensen-ishii}. We apply it twice: if $|\xi|$ is large, the vector
	$\xi+\nabla u$ stays away from zero, and the equation behaves as a uniformly
	elliptic one; if, on the other hand, $|\xi|$ is bounded, the equation may
	still degenerate, but a weaker penalization of the form $|x-y|^\beta$ yields a
	uniform H\"older estimate.
	
	\begin{lemma}[Lipschitz estimate for large-shift]\label{reg holder lipsc}
		Assume \ref{hipoteseF} -- \ref{hip2}, and let $u\in \mathrm{C}^{0}(B_1)$ be a normalized viscosity solution of \eqref{eq:shifted}.
		Fix $0<r<1/2$. There exists a constant
		\[
		L_0=L_0\bigl(n,\lambda,\Lambda,p^{+},r,\|a\|_\infty,\|f\|_\infty\bigr)>0
		\]
		such that, if $|\xi|\ge L_0$, then $u\in \mathrm{C}^{0,1}_{\mathrm{loc}}(B_1)$ and
		\[
		[u]_{\mathrm{C}^{0,1}(B_r(x_0))}
		\le
		C\bigl(n,\lambda,\Lambda,p^{+},r,\|a\|_\infty,\|f\|_\infty\bigr)
		\]
		for every $x_0\in B_{1/2}$ such that $B_r(x_0)\subset B_1$.
		
		In particular, under the smallness regime
		$\|a\|_\infty+\|f\|_\infty\le1$, both $L_0$ and $C$ depend only on
		$n,\lambda,\Lambda,p^{+}$ and $r$.
	\end{lemma}
	
	\begin{proof}
		Fix $x_0\in B_{1/2}$ and $0<r<1/2$ such that $B_r(x_0)\subset B_1$.
		We claim that, for suitable constants $L_1,L_2>0$,
		\begin{equation}\label{eq:Lipschitz-doubling-claim}
			\sup_{B_r(x_0)\times B_r(x_0)}
			\Big\{
			u(x)-u(y)-L_1v(|x-y|)
			-L_2\bigl(|x-x_0|^2+|y-x_0|^2\bigr)
			\Big\}
			\le 0,
		\end{equation}
		where, for $0\le s\le s_0$,
		\[
		v(s)=s-\omega_0s^{3/2}
		\]
		and $\omega_0>0$ is chosen so that
		\[
		s_0:=\left(\frac{2}{3\omega_0}\right)^2\ge 1.
		\]
		Since $|x-y|\le 2r<1$, only this first branch of $v$ is used. Notice that
		\[
		v'(s)=1-\frac{3\omega_0}{2} s^{1/2}>0
		\quad \text{and} \quad
		v''(s)=-\frac{3\omega_0}{4} s^{-1/2}\le -\frac{3\omega_0}{4}.
		\]
		First,
		we choose 
		\[
		L_2:=\left(\frac{8}{r}\right)^2
		\]
		and the constant $L_1$ is to be chosen later, depending on $L_2$ and on the data. We
		argue by contradiction and assume the supremum in
		\eqref{eq:Lipschitz-doubling-claim} is positive. Let $(\bar x,\bar y)$ be a point
		where this positive maximum is attained, and set
		\[
		\Phi(x,y):=
		L_1v(|x-y|)
		+
		L_2\bigl(|x-x_0|^2+|y-x_0|^2\bigr).
		\]
		Thus,
		\[
		u(\bar x)-u(\bar y)-\Phi(\bar x,\bar y)>0.
		\]
		Since $\|u\|_{\mathrm{L}^{\infty}(B_1)}\le1$, we have
		\[
		L_2\bigl(|\bar x-x_0|^2+|\bar y-x_0|^2\bigr)\le 2.
		\]
		With our choice of $L_2$, this gives
		\[
		|\bar x-x_0|+|\bar y-x_0|\le \frac r2.
		\]
		Hence $\bar x,\bar y\in B_r(x_0)$ and the maximum point is interior. Moreover
		$\bar x\neq \bar y$, since otherwise maximum would be zero.
		
		By \Cref{prop:jensen-ishii}, applied at this maximum point
		$(\bar x,\bar y)$ of
		$u(x)-u(y)-\Phi(x,y)$,
		there exist symmetric matrices $X,Y\in\mathcal S(n)$ such that the equation may
		be tested from above at $\bar x$ with the pair $\left(\nabla_x\Phi(\bar x,\bar y),X\right)$
		and from below at $\bar y$ with the pair $
		\left(-\nabla_y\Phi(\bar x,\bar y),Y\right).$ 
		Thus, setting
		\[
		\eta_{\bar x}:=\nabla_x\Phi(\bar x,\bar y)
		\quad\text{and}\quad
		\eta_{\bar y}:=-\nabla_y\Phi(\bar x,\bar y),
		\]
		we have
		\[
		\eta_{\bar x}
		=
		L_1v'(|\bar x-\bar y|)
		\frac{\bar x-\bar y}{|\bar x-\bar y|}
		+
		2L_2(\bar x-x_0),
		\]
		and
		\[
		\eta_{\bar y}
		=
		L_1v'(|\bar x-\bar y|)
		\frac{\bar x-\bar y}{|\bar x-\bar y|}
		-
		2L_2(\bar y-x_0).
		\]
		Since $u$ is a viscosity solution, this yields
		\[
		|\eta_{\bar x}+\xi|^{p(\bar x)}F(X)
		+
		a(\bar x)|\eta_{\bar x}+\xi|^{q(\bar x)}
		\ge f(\bar x),
		\]
		and
		\[
		|\eta_{\bar y}+\xi|^{p(\bar y)}F(Y)
		+
		a(\bar y)|\eta_{\bar y}+\xi|^{q(\bar y)}
		\le f(\bar y).
		\]
		In addition, we have the matrix inequality
		\[
		\begin{pmatrix}
			X & 0\\
			0 & -Y
		\end{pmatrix}
		\le
		\begin{pmatrix}
			Z & -Z\\
			-Z & Z
		\end{pmatrix}
		+
		(2L_2+\kappa)I,
		\]
		with $Z$ being the Hessian of the radial term
		\[
		Z
		:=
		L_1\left[
		\frac{v'(|\bar x-\bar y|)}{|\bar x-\bar y|}I
		+
		\left(
		v''(|\bar x-\bar y|)
		-
		\frac{v'(|\bar x-\bar y|)}{|\bar x-\bar y|}
		\right)
		\frac{(\bar x-\bar y)\otimes(\bar x-\bar y)}
		{|\bar x-\bar y|^2}
		\right],
		\] and
		where $\kappa>0$ can be taken arbitrarily small.
		Applying this to vectors of the form $(z,z)$ gives
		\[
		\langle (X-Y)z,z\rangle
		\le
		(4L_2+2\kappa)|z|^2.
		\]
		Thus, all eigenvalues of $X-Y$ are bounded above by $4L_2+2\kappa$. On the other hand, for the vector
		\[
		\left(
		\frac{\bar x-\bar y}{|\bar x-\bar y|},
		\frac{\bar y-\bar x}{|\bar x-\bar y|}
		\right)
		\]
		we obtain
		\[
		\left\langle
		(X-Y)\frac{\bar x-\bar y}{|\bar x-\bar y|},
		\frac{\bar x-\bar y}{|\bar x-\bar y|}
		\right\rangle
		\le
		4L_2+2\kappa+4L_1v''(|\bar x-\bar y|).
		\]
		Hence, at least one eigenvalue of $X-Y$ is bounded above by
		\[
		4L_2+2\kappa+4L_1v''(|\bar x-\bar y|).
		\]
		Using the definition of the Pucci extremal operator, we obtain
		\begin{equation}\label{eq:Pucci-large-shift}
			\mathcal P^+_{\lambda,\Lambda}(X-Y)
			\le
			2(\lambda+\Lambda(n-1))(2L_2+\kappa)
			+
			4\lambda L_1v''(|\bar x-\bar y|).
		\end{equation}
		Since $v''\le -3\omega_0/4$, choosing $\kappa\le1$ yields
		\begin{equation}\label{eq:Pucci-negative-large-shift}
			\mathcal P^+_{\lambda,\Lambda}(X-Y)
			\le
			2(\lambda+\Lambda(n-1))(2L_2+1)
			-
			3\lambda\omega_0L_1.
		\end{equation}
		Now, we use the viscosity inequalities. Since $u$ is a subsolution at $\bar x$
		and a supersolution at $\bar y$, we have
		\begin{equation}\label{eq:visc-large-shift}
			|\eta_{\bar x}+\xi|^{p(\bar x)}F(X)
			+
			a(\bar x)|\eta_{\bar x}+\xi|^{q(\bar x)}
			\ge f(\bar x),
		\end{equation}
		and
		\begin{equation}\label{eq:visc-large-shift-y}
			|\eta_{\bar y}+\xi|^{p(\bar y)}F(Y)
			+
			a(\bar y)|\eta_{\bar y}+\xi|^{q(\bar y)}
			\le f(\bar y).
		\end{equation}
		Moreover,
		\[
		|\eta_{\bar x}|+|\eta_{\bar y}|
		\le
		2L_1+4L_2.
		\]
		We now choose $L_0:=4(2L_1+4L_2+1)$.
		If $|\xi|\ge L_0$ and $e:=\xi/|\xi|$, then
		\[
		\left|e+\frac{\eta_{\bar x}}{|\xi|}\right|, \ 
		\left|e+\frac{\eta_{\bar y}}{|\xi|}\right|
		\in \left[\frac12,\frac32\right].
		\]
		Dividing \eqref{eq:visc-large-shift} and \eqref{eq:visc-large-shift-y} by the
		corresponding powers of $|\xi|$, and using $0\le q\le p$,
		\[
		F(X)
		\le
		C\bigl(\|f\|_\infty+\|a\|_\infty\bigr),
		\]
		from the first inequality, and
		\[
		F(Y)
		\ge
		-C\bigl(\|f\|_\infty+\|a\|_\infty\bigr),
		\]
		from the second one, where $C=C(p^{+})$.
		By uniform ellipticity,
		\[
		F(X)-F(Y)\le \mathcal P^+_{\lambda,\Lambda}(X-Y)
		\] and,
		consequently,
		\[
		-C\bigl(\|f\|_\infty+\|a\|_\infty\bigr)
		\le
		\mathcal P^+_{\lambda,\Lambda}(X-Y).
		\]
		Combining this with \eqref{eq:Pucci-negative-large-shift} yields
		\[
		3\lambda\omega_0L_1
		\le
		C\bigl(\|f\|_\infty+\|a\|_\infty\bigr)
		+
		2(\lambda+\Lambda(n-1))(2L_2+1).
		\]
		Thus, if $L_1$ is chosen so that
		\[
		L_1>
		\frac{
			C\bigl(\|f\|_\infty+\|a\|_\infty\bigr)
			+
			2(\lambda+\Lambda(n-1))(2L_2+1)
		}{3\lambda\omega_0},
		\]
		we reach a contradiction. Hence \eqref{eq:Lipschitz-doubling-claim} holds.
		
		Finally, since $v(s)\ge c s$ for $0\le s\le 2r$, the claim implies
		\[
		u(x)-u(y)\le C L_1|x-y|
		\]
		for all $x,y\in B_{r/2}(x_0)$. Changing the roles of $x$ and $y$ provides the desired Lipschitz
		estimate.
	\end{proof}

	The preceding lemma covers the so-called nondegenerate regime, where the large shift
	$\xi$ keeps the effective gradient away from zero. We now treat the complementary
	case, in which $|\xi|$ is bounded. In this regime, we do not expect the same
	Lipschitz control, but the same method still yields a
	uniform H\"older estimate.

	\begin{lemma}[H\"older estimate in the bounded-shift regime]\label{lem:holder-small-shift}
		Assume \ref{hipoteseF} -- \ref{hip2}, and let $u\in \mathrm{C}^{0}(B_1)$ be a normalized viscosity solution of \eqref{eq:shifted}.
		Fix $0<r<1/2$, and let $L_0$ be as in Lemma \ref{reg holder lipsc}. If
		$|\xi|\le L_0$, then there exist $\beta\in(0,1)$ and
		\[
		C=C\bigl(n,\lambda,\Lambda,p^{+},r,L_0,\|a\|_\infty,\|f\|_\infty\bigr)>0
		\]
		such that
		\[
		[u]_{\mathrm{C}^{0,\beta}(B_r(x_0))}\le C
		\]
		for every $x_0\in B_{1/2}$ such that $B_r(x_0)\subset B_1$.
		
		Under the smallness regime
		$\|a\|_\infty+\|f\|_\infty\le1$, the constants depend only on
		$n,\lambda,\Lambda,p^{+}$ and $r$.
	\end{lemma}

	\begin{proof}
		Fix $\beta\in(0,1)$ to be chosen below. We prove that, for $L_1$ large enough,
		\begin{equation}\label{eq:holder-doubling-claim}
			\sup_{B_r(x_0)\times B_r(x_0)}
			\Big\{
			u(x)-u(y)-L_1|x-y|^\beta
			-L_2\bigl(|x-x_0|^2+|y-x_0|^2\bigr)
			\Big\}
			\le 0,
		\end{equation}
		where, as before, $L_2$ is the same as in the previous proof.
		Assume, by contradiction, that the supremum in \eqref{eq:holder-doubling-claim}
		is positive and let $(\bar x,\bar y)$ be an interior maximum point. We have $\bar x\ne\bar y$ and the same
		localization argument as in Lemma \ref{reg holder lipsc} gives
		\[
		|\bar x-x_0|+|\bar y-x_0|\le \frac r2.
		\]
		
		Applying \Cref{prop:jensen-ishii} with
		\[
		\Phi(x,y)
		=
		L_1|x-y|^\beta
		+
		L_2\bigl(|x-x_0|^2+|y-x_0|^2\bigr),
		\]
		at the maximum point $(\bar x,\bar y)$ of $
		u(x)-u(y)-\Phi(x,y)$,
		we obtain symmetric matrices $X,Y\in\mathcal S(n)$ such that the equation may be
		tested from above at $\bar x$ with first-order part $
		\eta_{\bar x}:=\nabla_x\Phi(\bar x,\bar y)$,
		and from below at $\bar y$ with first-order part $
		\eta_{\bar y}:=-\nabla_y\Phi(\bar x,\bar y)$.
		Since $\bar x\neq \bar y$, these vectors are explicitly given by
		\[
		\eta_{\bar x}
		=
		L_1\beta|\bar x-\bar y|^{\beta-2}(\bar x-\bar y)
		+
		2L_2(\bar x-x_0),
		\]
		and
		\[
		\eta_{\bar y}
		=
		L_1\beta|\bar x-\bar y|^{\beta-2}(\bar x-\bar y)
		-
		2L_2(\bar y-x_0).
		\]
		Moreover, the matrix inequality yields
		\begin{equation}\label{eq:Pucci-small-shift}
			\mathcal P^+_{\lambda,\Lambda}(X-Y)
			\le
			C_0L_2
			-
			c_0L_1|\bar x-\bar y|^{\beta-2},
		\end{equation}
		where $C_0=C_0(n,\lambda,\Lambda)$ and $
		c_0=4\lambda\beta(1-\beta)>0$.
		At the maximum point,
		\[
		L_1|\bar x-\bar y|^\beta\le 2,
		\]
		and hence $|\bar x-\bar y|\to0$ as $L_1\to\infty$. On the other hand,
		\[
		|\eta_{\bar x}|
		\ge
		L_1\beta|\bar x-\bar y|^{\beta-1}-2L_2r,
		\]
		and similarly for $\eta_{\bar y}$. Since $\beta-1<0$, the first term becomes large
		when $|\bar x-\bar y|$ is small. Hence, choosing $L_1$ large enough depending on
		$r,L_2,L_0$ and $\beta$, we ensure
		\begin{equation}\label{eq:jet-lower-bound}
			|\eta_{\bar x}+\xi|\ge 1
			\quad\text{and}\quad
			|\eta_{\bar y}+\xi|\ge 1.
		\end{equation}
		
		Using the viscosity inequalities as in the previous lemma, we get
		\[
		|\eta_{\bar x}+\xi|^{p(\bar x)}F(X)
		+
		a(\bar x)|\eta_{\bar x}+\xi|^{q(\bar x)}
		\ge f(\bar x),
		\]
		and
		\[
		|\eta_{\bar y}+\xi|^{p(\bar y)}F(Y)
		+
		a(\bar y)|\eta_{\bar y}+\xi|^{q(\bar y)}
		\le f(\bar y).
		\]
		Because $|\eta_{\bar x}+\xi|,|\eta_{\bar y}+\xi|\ge1$ and $0\le q\le p$, we obtain
		\[
		F(X)\le C\bigl(\|f\|_\infty+\|a\|_\infty\bigr),
		\]
		and
		\[
		F(Y)\ge -C\bigl(\|f\|_\infty+\|a\|_\infty\bigr),
		\]
		with $C=C(p^{+})$. Therefore,
		\[
		-C\bigl(\|f\|_\infty+\|a\|_\infty\bigr)
		\le
		F(X)-F(Y)
		\le
		\mathcal P^+_{\lambda,\Lambda}(X-Y).
		\]
		Using \eqref{eq:Pucci-small-shift}, we find
		\[
		c_0L_1|\bar x-\bar y|^{\beta-2}
		\le
		C_0L_2
		+
		C\bigl(\|f\|_\infty+\|a\|_\infty\bigr).
		\]
		But, since $|\bar x-\bar y|\le (2/L_1)^{1/\beta}$, we have
		\[
		L_1|\bar x-\bar y|^{\beta-2}
		=
		L_1|\bar x-\bar y|^\beta|\bar x-\bar y|^{-2}
		\ge
		cL_1^{2/\beta}.
		\]
		Thus, the left-hand side becomes arbitrarily large as $L_1\to\infty$, while the
		right-hand side is fixed. This contradiction proves \eqref{eq:holder-doubling-claim}.
		
		Consequently,
		\[
		u(x)-u(y)\le L_1|x-y|^\beta
		\]
		for $x,y\in B_{r/2}(x_0)$. Exchanging $x$ and $y$ gives the desired
		$\mathrm{C}^{0,\beta}$ estimate.
	\end{proof}

	The two lemmas above exhaust all possible sizes of the shift $\xi$. Therefore,
	after possibly lowering the exponent, we obtain a H\"older estimate that is uniform
	with respect to $\xi$, which provides the compactness needed in the next section.

	\begin{corollary}[Uniform H\"older estimates]\label{cor:holder-compactness}
		Assume \ref{hipoteseF} -- \ref{hip2}, and let $u\in \mathrm{C}^{0}(B_1)$ be a normalized viscosity solution of \eqref{eq:shifted}.
		Assume further that
		\[
		\|a\|_{\mathrm{L}^{\infty}(B_1)}+\|f\|_{\mathrm{L}^{\infty}(B_1)}\le 1.
		\]
		Then, there exist $\beta\in(0,1)$ and $C>0$, depending only on
		$n,\lambda,\Lambda,p^{+}$, such that
		\[
		\|u\|_{\mathrm{C}^{0,\beta}(B_{1/2})}\le C.
		\]
	\end{corollary}
	
	\begin{proof}
		If $|\xi|\ge L_0$, the conclusion follows from \Cref{reg holder lipsc}.
		If $|\xi|<L_0$, it follows from Lemma \ref{lem:holder-small-shift}. Taking the
		smaller of the two exponents and the larger of the two constants gives the
		uniform estimate.
	\end{proof}

	\section{Approximation lemma}\label{sec_compactness}
	
	The purpose of this section is to convert the uniform H\"older estimates
	obtained above into a tangential approximation result. Roughly speaking, if the
	source term and the coefficient of the Hamiltonian perturbation are sufficiently
	small; then, normalized solutions of the shifted equation are uniformly close, in
	smaller balls, to solutions of homogeneous uniformly elliptic equations.
	
	The main point is the following stability statement. Along a convergent sequence
	of equations, the terms involving $a_j$ and $f_j$ vanish. Whenever the effective
	gradient does not vanish, one may divide by the factor
	$|\nabla u_j+\xi_j|^{p_j(x)}$ and pass directly to the limit. The only delicate
	case occurs at contact points where the limiting effective gradient vanishes. In
	that case, a small perturbation of the test function in a positive eigenspace of
	the Hessian restores a nonzero effective gradient and allows for the same conclusion.
	
	\begin{proposition}[Stability]\label{propHolde}
		Let $\{u_j\}_{j\in\mathbb N}\subset \mathrm{C}^{0}(B_1)\cap \mathrm{L}^{\infty}(B_1)$ be a sequence
		of viscosity solutions of
		\begin{equation}\label{eq:stability-sequence}
			|\nabla u_j+\xi_j|^{p_j(x)}F_j(\nabla^{2}u_j)
			+
			a_j(x)|\nabla u_j+\xi_j|^{q_j(x)}
			=
			f_j(x)
			\quad\text{in }B_1,
		\end{equation}
		where $\xi_j\in\mathbb R^n$. Suppose that:
		\begin{enumerate}[label=$(\roman*)$]
			\item $F_j:\mathcal S(n)\to\mathbb R$ is uniformly
			$(\lambda,\Lambda)$-elliptic and $F_j(0)=0$ for every $j$;
			
			\item $F_j\to \mathcal F$ locally uniformly in
			$\mathcal S(n)$, where $\mathcal F$ is uniformly
			$(\lambda,\Lambda)$-elliptic;
			
			\item
			$
			\|f_j\|_{\mathrm{L}^{\infty}(B_1)}
			+
			\|a_j\|_{\mathrm{L}^{\infty}(B_1)}
			\le 1/j$;
			
			\item there exists $c_0>0$ such that, for every $x\in B_1$,
			\[
			0\le q_j(x)\le p_j(x)\le c_0;
			\]
			
			\item $u_j\to u$ locally uniformly in $B_1$.
		\end{enumerate}
		Then, $u$ is a viscosity solution of
		\begin{equation}\label{OperatorProp}
			\mathcal F(\nabla^{2}u)=0
			\quad\text{in }B_{3/4}.
		\end{equation}
	\end{proposition}
	
	\begin{proof}
		We prove first that $u$ is a viscosity supersolution of
		\eqref{OperatorProp}. Let $\varphi\in C^2(B_{3/4})$ touch $u$ from below at
		$x_0\in B_{3/4}$. We need to prove
		\[
		\mathcal F(\nabla^{2}\varphi(x_0))\le0.
		\]
		We may assume that the contact is
		strict in a small ball $B_\rho(x_0)\subset\subset B_{3/4}$. Namely,
		\[
		u(x_0)=\varphi(x_0),
		\quad\text{and}\quad
		u(x)>\varphi(x),
		\quad\text{for }x\in \overline{B_\rho(x_0)}\setminus\{x_0\}.
		\]
		Since $u_j\to u$ uniformly in $\overline{B_\rho(x_0)}$, there exist points
		$x_j\in B_\rho(x_0)$, with $x_j\to x_0$, such that $u_j-\varphi$
		has a local minimum at $x_j$. Hence, since $u_j$ is a viscosity solution,
		\begin{equation}\label{eq:visc-stability-main}
			|\nabla\varphi(x_j)+\xi_j|^{p_j(x_j)}
			F_j(\nabla^{2}\varphi(x_j))
			+
			a_j(x_j)|\nabla\varphi(x_j)+\xi_j|^{q_j(x_j)}
			\le
			f_j(x_j).
		\end{equation}
		
		Now, we distinguish two cases.
		
		\medskip
		
		\noindent\textit{Case 1:} the sequence $\{\xi_j\}$ is unbounded.
		Passing to a subsequence, we may assume $|\xi_j|\to+\infty$. Since
		$\nabla\varphi(x_j)$ is bounded, then, for $j$ large enough, we have
		\[
		|\nabla\varphi(x_j)+\xi_j|\ge \frac12|\xi_j|.
		\]
		Dividing \eqref{eq:visc-stability-main} by
		$|\nabla\varphi(x_j)+\xi_j|^{p_j(x_j)}$, we obtain
		\[
		F_j(\nabla^{2}\varphi(x_j))
		+
		\frac{a_j(x_j)}
		{|\nabla\varphi(x_j)+\xi_j|^{p_j(x_j)-q_j(x_j)}}
		\le
		\frac{f_j(x_j)}
		{|\nabla\varphi(x_j)+\xi_j|^{p_j(x_j)}}.
		\]
		Because $0\le q_j\le p_j\le c_0$ and
		$\|a_j\|_\infty+\|f_j\|_\infty\to0$, both terms involving $a_j$ and $f_j$
		converge to zero. Using the local uniform convergence $F_j\to\mathcal F$ and
		the convergence $\nabla^{2}\varphi(x_j)\to \nabla^{2}\varphi(x_0)$, we conclude
		\[
		\mathcal F(\nabla^{2}\varphi(x_0))\le0.
		\]
		
		\medskip
		
		\noindent\textit{Case 2:} the sequence $\{\xi_j\}$ is bounded.
		Passing to a subsequence, we may assume
		$\xi_j\to \xi_\infty$.
		Set
		\[
		b:=\nabla\varphi(x_0)
		\quad\text{and}\quad
		M:=\nabla^{2}\varphi(x_0).
		\]
		If $|b+\xi_\infty|>0$, then
		\[
		|\nabla\varphi(x_j)+\xi_j|\ge c>0
		\]
		for all sufficiently large $j$. Dividing \eqref{eq:visc-stability-main} by
		$|\nabla\varphi(x_j)+\xi_j|^{p_j(x_j)}$ gives
		\[
		F_j(\nabla^{2}\varphi(x_j))
		+
		\frac{a_j(x_j)}
		{|\nabla\varphi(x_j)+\xi_j|^{p_j(x_j)-q_j(x_j)}}
		\le
		\frac{f_j(x_j)}
		{|\nabla\varphi(x_j)+\xi_j|^{p_j(x_j)}}.
		\]
		Since $0\le p_j-q_j\le c_0$, the denominators are bounded away from zero in
		a uniform way. Letting $j\to\infty$, we get
		\[
		\mathcal F(M)\le0.
		\]
		It remains to consider the critical case
		$b+\xi_\infty=0$.
		We prove the desired inequality by contradiction. Assume
		$\mathcal F(M)>0$.
		Since $\mathcal F(0)=0$ and $\mathcal F$ is degenerate elliptic, this implies
		that $M$ has at least one positive eigenvalue: indeed, if $M\le0$, then
		\[
		0=\mathcal F(0)\ge \mathcal F(M),
		\]
		contradicting $\mathcal F(M)>0$.
		Next, let $E\subset\mathbb R^n$ be the subspace spanned by all eigenvectors of
		$M$ associated with positive eigenvalues, and let $P_E$ denote the
		orthogonal projection onto $E$. We fix $\gamma>0$ and consider the
		perturbed function
		\[
		\varphi_\gamma(x):=\varphi(x)+\gamma |P_E(x-x_0)|.
		\]
		Although $\varphi_\gamma$ is not smooth on the set
		$\{P_E(x-x_0)=0\}$, it is continuous. Since $\varphi$ touches $u$ strictly
		from below at $x_0$ and $\gamma>0$ is fixed sufficiently small, the functions $u_j-\varphi_\gamma$
		have local minimum points $x_j^\gamma\in B_\rho(x_0)$, with
		$x_j^\gamma\to x_0$.

		Now, in subcases $(a)$ and $(b)$ below, we distinguish whether the contact points lie on a differentiability point of
		$|P_E(x-x_0)|$ or not.
		
		\smallskip
		
		$(a)$ For infinitely many $j$, we have $P_E(x_j^\gamma-x_0)=0$.
		Passing to a subsequence, if necessary, we assume this holds for every $j$. Choose any
		$e\in E$ with $|e|=1$, and define the smooth function
		\[
		\psi_{\gamma,e}(x)
		:=
		\varphi(x)+\gamma \, P_E(x-x_0) \cdot e.
		\]
		Since $P_E(x-x_0) \cdot e\le |P_E(x-x_0)|$,
		for every $x$, and equality holds at $x=x_j^\gamma$, the function
		$\psi_{\gamma,e}$ also touches $u_j$ from below at $x_j^\gamma$, up to an additive constant. Then,
		\[
		|\nabla\psi_{\gamma,e}(x_j^\gamma)+\xi_j|^{p_j(x_j^\gamma)}
		F_j(\nabla^{2}\psi_{\gamma,e}(x_j^\gamma))
		+
		a_j(x_j^\gamma)
		|\nabla\psi_{\gamma,e}(x_j^\gamma)+\xi_j|^{q_j(x_j^\gamma)}
		\le
		f_j(x_j^\gamma).
		\]
		But
		\[
		\nabla\psi_{\gamma,e}(x_j^\gamma)
		=
		\nabla\varphi(x_j^\gamma)+\gamma e
		\quad\text{and}\quad
		\nabla^{2}\psi_{\gamma,e}(x_j^\gamma)
		=
		\nabla^{2}\varphi(x_j^\gamma),
		\] and
		since
		\[
		\nabla\varphi(x_j^\gamma)+\xi_j\to b+\xi_\infty=0,
		\]
		we have, for all sufficiently large $j$,
		\[
		|\nabla\psi_{\gamma,e}(x_j^\gamma)+\xi_j|
		\ge \frac{\gamma}{2}.
		\]
		Dividing by
		$|\nabla\psi_{\gamma,e}(x_j^\gamma)+\xi_j|^{p_j(x_j^\gamma)}$, we obtain
		\[
		F_j(\nabla^{2}\varphi(x_j^\gamma))
		+
		\frac{a_j(x_j^\gamma)}
		{|\nabla\psi_{\gamma,e}(x_j^\gamma)+\xi_j|^{p_j(x_j^\gamma)-q_j(x_j^\gamma)}}
		\le
		\frac{f_j(x_j^\gamma)}
		{|\nabla\psi_{\gamma,e}(x_j^\gamma)+\xi_j|^{p_j(x_j^\gamma)}}.
		\]
		The terms involving $a_j$ and $f_j$ converge to zero, because by construction
		\[
		\frac{1}
		{|\nabla\psi_{\gamma,e}(x_j^\gamma)+\xi_j|^{p_j(x_j^\gamma)-q_j(x_j^\gamma)}}
		\le
		\left(\frac2\gamma\right)^{c_0},
		\]
		and similarly for the right-hand side. Letting $j\to\infty$, we conclude
		\[
		\mathcal F(M)\le0,
		\] which
		contradicts the assumption $\mathcal F(M)>0$.
		
		\smallskip
		
		$(b)$ For infinitely many $j$, we have $P_E(x_j^\gamma-x_0)\neq0$.
		Up to passing to a subsequence, we assume this holds for every $j$. Then
		$\varphi_\gamma$ is smooth near $x_j^\gamma$ and may be used directly as a test
		function. Set
		\[
		\nu_j^\gamma
		:=
		\frac{P_E(x_j^\gamma-x_0)}
		{|P_E(x_j^\gamma-x_0)|}.
		\]
		At $x_j^\gamma$ we have
		\[
		\nabla\varphi_\gamma(x_j^\gamma)
		=
		\nabla\varphi(x_j^\gamma)+\gamma\nu_j^\gamma
		\quad
		\text{and}
		\quad 
		\nabla^{2}\varphi_\gamma(x_j^\gamma)
		=
		\nabla^{2}\varphi(x_j^\gamma)+\gamma B_j,
		\]
		where
		\[
		B_j
		:=
		\frac{1}{|P_E(x_j^\gamma-x_0)|}
		\left(P_E-\nu_j^\gamma\otimes\nu_j^\gamma\right).
		\]
		Notice that $B_j\ge0$. Then, by ellipticity,
		\[
		F_j(\nabla^{2}\varphi(x_j^\gamma))
		\le
		F_j(\nabla^{2}\varphi(x_j^\gamma)+\gamma B_j).
		\]
		The viscosity supersolution inequality gives
		\[
		|\nabla\varphi_\gamma(x_j^\gamma)+\xi_j|^{p_j(x_j^\gamma)}
		F_j(\nabla^{2}\varphi(x_j^\gamma)+\gamma B_j)
		+
		a_j(x_j^\gamma)
		|\nabla\varphi_\gamma(x_j^\gamma)+\xi_j|^{q_j(x_j^\gamma)}
		\le
		f_j(x_j^\gamma)
		\] and since,
		moreover,
		\[
		\nabla\varphi(x_j^\gamma)+\xi_j\to0,
		\quad\text{and}
		\quad
		|\nu_j^\gamma|=1,
		\]
		we obtain, for $j$ large enough,
		\[
		|\nabla\varphi_\gamma(x_j^\gamma)+\xi_j|
		\ge \frac{\gamma}{2}.
		\]
		Dividing by
		$|\nabla\varphi_\gamma(x_j^\gamma)+\xi_j|^{p_j(x_j^\gamma)}$, we have
		\[
		F_j(\nabla^{2}\varphi(x_j^\gamma)+\gamma B_j)
		+
		\frac{a_j(x_j^\gamma)}
		{|\nabla\varphi_\gamma(x_j^\gamma)+\xi_j|^{p_j(x_j^\gamma)-q_j(x_j^\gamma)}}
		\le
		\frac{f_j(x_j^\gamma)}
		{|\nabla\varphi_\gamma(x_j^\gamma)+\xi_j|^{p_j(x_j^\gamma)}}.
		\]
		Using the monotonicity estimate above, given by the ellipticity, this implies
		\[
		F_j(\nabla^{2}\varphi(x_j^\gamma))
		+
		\frac{a_j(x_j^\gamma)}
		{|\nabla\varphi_\gamma(x_j^\gamma)+\xi_j|^{p_j(x_j^\gamma)-q_j(x_j^\gamma)}}
		\le
		\frac{f_j(x_j^\gamma)}
		{|\nabla\varphi_\gamma(x_j^\gamma)+\xi_j|^{p_j(x_j^\gamma)}}.
		\]
		As in the previous subcase, the terms involving $a_j$ and $f_j$ converge to
		zero. Passing to the limit, we obtain $\mathcal F(M)\le0$,
		again contradicting $\mathcal F(M)>0$.
		
		The critical case is therefore impossible unless $\mathcal F(M)\le0$
		and it follows that $u$ is a viscosity supersolution of \eqref{OperatorProp}.
		
		\smallskip
		
		It remains to prove that $u$ is a viscosity subsolution of
		\eqref{OperatorProp}. But this follows by applying the previous argument to
		$\widetilde u_{j}:=-u_j$ and $\widetilde F_j(M):=-F_j(-M)$.  
%		Indeed, $v_j\to -u$ locally uniformly, and the corresponding
%		operators
%		\[
%		\widetilde F_j(M):=-F_j(-M)
%		\]
%		are uniformly $(\lambda,\Lambda)$-elliptic, satisfy
%		$\widetilde F_j(0)=0$, and converge locally uniformly to
%		\[
%		\widetilde{\mathcal F}(M):=-\mathcal F(-M).
%		\]
%		The shifted gradients are transformed by replacing $\xi_j$ with $-\xi_j$.
%		The same proof shows that $-u$ is a viscosity supersolution of
%		\[
%		\widetilde{\mathcal F}(\nabla^{2}(-u))=0,
%		\]
%		which is equivalent to saying that $u$ is a viscosity subsolution of
%		\[
%		\mathcal F(\nabla^{2}u)=0.
%		\]
		Therefore, $u$ is a viscosity solution of \eqref{OperatorProp}.
	\end{proof}
	
	We can now prove the approximation lemma. It says that, in the smallness regime,
	normalized solutions of the shifted equation are close to solutions of homogeneous uniformly
	elliptic problems. This is the compactness input for the improvement-of-flatness
	iteration.
	
	\begin{lemma}[Approximation lemma]\label{lem2-4}
		Assume \ref{hipoteseF} -- \ref{hip2}, and let $u\in \mathrm{C}^{0}(B_1)$ be a normalized viscosity solution of \eqref{eq:shifted}, where $\xi\in\mathbb R^n$ is arbitrary. Given $\varepsilon>0$, there exists $
		\delta=\delta(n,\lambda,\Lambda,p^{+},\varepsilon)>0$
		such that, if
		\[
		\|a\|_{\mathrm{L}^{\infty}(B_1)}+\|f\|_{\mathrm{L}^{\infty}(B_1)}\le\delta;
		\]
		then, there exist a uniformly $(\lambda,\Lambda)$-elliptic operator
		$\mathcal F:\mathcal S(n)\to\mathbb R$, with $\mathcal F(0)=0$, and a
		viscosity solution $h$ of $\mathcal F(\nabla^{2}h)=0$ in $B_{3/4}$,
		such that
		\[
		\|u-h\|_{\mathrm{L}^{\infty}(B_{1/2})}\le \varepsilon.
		\]
		Moreover, for a constant $C$ that depends only on $n,\lambda,\Lambda$, we have
		\[
		\|h\|_{\mathrm{C}^{1,\alpha_0}(B_{1/2})}\le C.
		\]
	\end{lemma}
	
	\begin{proof}
		We argue by contradiction. Then, there exist $\varepsilon_0>0$, sequences of functions $\{a_j\},\{p_j\}$, $\{q_j\}$, $\{f_j\}$, $\{u_j\}$, $\{F_j\}$ and of vectors $\{\xi_j\}$ satisfying 
		\begin{enumerate}[label = $(\roman*)$]
			\item $F_j$ is uniformly $(\lambda,\Lambda)$-elliptic;
			\smallskip
			\item $\|f_j\|_{\mathrm{L}^{\infty}(B_1)}+ \|a_j\|_{\mathrm{L}^{\infty}(B_1)}\leq \frac{1}{j}$ and $f_j,a_j\in \mathrm{C}^{0}(B_1)$;
			\smallskip
			\item $0\leq q_j(x)\leq p_j(x)$, $p_j(x)\leq \sup_{x\in B_1}p(x)$ and $p_j,q_j\in \mathrm{C}^{0}(B_1)$;
			\smallskip
			\item $u_j\in \mathrm{C}^{0}(B_1)$ and $\|u_j\|_{\mathrm{L}^{\infty}(B_1)}\leq 1$,
		\end{enumerate}
		and such that
		\begin{equation}
			\label{2-12}
			|\nabla u_j+\xi_j|^{p_j(x)}F_j(\nabla^2u_j)+a_j(x)|\nabla u_j+\xi_j|^{q_j(x)}=f_j(x)  \quad \text{in } B_1,
		\end{equation}
		but
		\begin{equation}
			\|u_j-h\|_{\mathrm{L}^{\infty}(B_{1/2})}>\varepsilon_0,
			\label{aprox-F-har}
		\end{equation}
		for all $h$ satisfying $\mathcal F(\nabla^2h)=0$, for any $(\lambda,\Lambda)$-elliptic operator $\mathcal F$ with $\mathcal F(0)=0$.

		By the H\"older compactness estimate from the previous section, there exist
		$\beta\in(0,1)$ and $C>0$, independent of $j$ and of $\xi_j$, such that
		\[
		\|u_j\|_{\mathrm{C}^{0,\beta}(B_{3/4})}\le C.
		\]
		Hence, by Arzel\`a--Ascoli, up to a subsequence,
		\[
		u_j\to \overline u,
		\quad\text{locally uniformly in }B_{3/4}.
		\]
		On the other hand, since the operators $F_j$ are uniformly elliptic and
		normalized by $F_j(0)=0$, they are locally equicontinuous in $\mathcal S(n)$.
		Thus, again up to a subsequence,
		\[
		F_j\to F_\infty,
		\quad\text{locally uniformly in }\mathcal S(n),
		\]
		where $F_\infty$ is uniformly $(\lambda,\Lambda)$-elliptic and
		$F_\infty(0)=0$.
		By \Cref{propHolde}, we have, in the viscosity sense,
		\[
		F_\infty(\nabla^{2}\overline u)=0
		\quad\text{in } B_{3/4}.
		\]
		Hence, $\overline u$ is an admissible homogeneous
		profile in \eqref{aprox-F-har}. Taking $h=\overline u$, we arrive at
		\[
		\|u_j-\overline u\|_{\mathrm{L}^{\infty}(B_{1/2})}>\varepsilon_0,
		\quad\text{for every }j,
		\]
		which contradicts the uniform convergence $u_j\to\overline u$ in
		$B_{1/2}$.
		
		Finally, by the universal $\mathrm{C}^{1,\alpha_0}$ regularity estimate for viscosity
		solutions of homogeneous uniformly elliptic equations,
		\[
		\|h\|_{\mathrm{C}^{1,\alpha_0}(B_{1/2})}
		\le
		C\|h\|_{\mathrm{L}^{\infty}(B_{3/4})}
		\le C,
		\]
		where $C$ depends only on $n,\lambda,\Lambda$. Therefore, the approximating profile
		has the claimed regularity estimate.
	\end{proof}

	\section{H\"older continuity of the gradient}\label{sec_c1alpha}
	
	In this section, we prove \Cref{teo_c1alpha}. Our argument adapts an
	improvement-of-flatness iteration. The approximation lemma, \Cref{lem2-4}, from the previous
	section, allows us to compare normalized solutions, at each scale, with solutions
	of homogeneous uniformly elliptic equations. Since those limiting profiles are
	universally $\mathrm{C}^{1,\alpha_0}$, each comparison improves the affine approximation
	of the original solution. Iterating this procedure yields uniform
	$\mathrm{C}^{1,\alpha}$ expansions at interior points.
	
	We first prove a one-step improvement lemma in the normalized regime.
	
	\begin{lemma}[One-step improvement of flatness]\label{lem:one-step-flatness}
		Assume \ref{hipoteseF} -- \ref{hip2}, and let $\alpha$ be as in \eqref{eqn:alpha-range}.
		Then, there exist constants $\rho\in(0,1/2)$ and $\delta=\delta(n,\lambda,\Lambda,p^{+},\alpha)>0$
		such that the following holds: if $u\in \mathrm{C}^{0}(B_1)$ is a viscosity solution of \eqref{eq:shifted} such that
		\[
		\|u\|_{\mathrm{L}^{\infty}(B_1)}\le1
		\quad \text{and}\quad 
		\|a\|_{\mathrm{L}^{\infty}(B_1)}+\|f\|_{\mathrm{L}^{\infty}(B_1)}\le\delta;
		\]
		then, there exists an affine function $\ell(x)=a_0+b_0\cdot x$
		such that
		\[
		\sup_{B_\rho}|u-\ell|\le \rho^{1+\alpha}
		\]
		and $|a_0|+|b_0|\le C$,
		where $C$ depends only on $n,\lambda,\Lambda$.
	\end{lemma}
	
	\begin{proof}
		Let $\varepsilon>0$ to be chosen later. By \Cref{lem2-4}, provided
		$\delta>0$ is small enough, there exist a uniformly $(\lambda,\Lambda)$-elliptic
		operator $\mathcal F$, with $\mathcal F(0)=0$, and a viscosity solution $h$ of
		\[
		\mathcal F(\nabla^2h)=0
		\quad\text{in }B_{3/4},
		\]
		such that
		\[
		\|u-h\|_{\mathrm{L}^{\infty}(B_{1/2})}\le \varepsilon.
		\]
		Since $h$ solves a homogeneous uniformly elliptic equation, the universal
		$\mathrm{C}^{1,\alpha_0}$ estimate gives
		\[
		\|h\|_{\mathrm{C}^{1,\alpha_0}(B_{1/2})}\le C,
		\]
		with $C=C(n,\lambda,\Lambda)$. In particular, $|h(0)|+|\nabla h(0)|\le C$
		and
		\[
		\sup_{B_\rho}
		\left|
		h(x)-h(0)-\nabla h(0)\cdot x
		\right|
		\le
		C\rho^{1+\alpha_0}.
		\]
		Define
		\[
		\ell(x):=h(0)+\nabla h(0)\cdot x.
		\]
		Then
		\[
		\sup_{B_\rho}|u-\ell|
		\le
		\sup_{B_\rho}|u-h|
		+
		\sup_{B_\rho}|h-\ell|
		\le
		\varepsilon+C\rho^{1+\alpha_0}.
		\]
		We choose $\rho\in(0,1/2)$ so small that $C\rho^{\alpha_0-\alpha}\le 1/2$,
		and next choose $\varepsilon:=\rho^{1+\alpha}/2$.
		Hence,
		\[
		\sup_{B_\rho}|u-\ell|
		\le
		\frac12\rho^{1+\alpha}
		+
		C\rho^{1+\alpha_0}
		\le
		\rho^{1+\alpha}.
		\]
		The bound $|a_0|+|b_0|\le C$ follows from the estimate
		$|h(0)|+|\nabla h(0)|\le C$.
	\end{proof}
	
	We now iterate the preceding improvement. The next proposition is the normalized
	version of the main estimate.
	
	\begin{proposition}[Normalized affine approximation]\label{prop:normalized-c1alpha}
		Assume \ref{hipoteseF} -- \ref{hip2}, let $\alpha$ be as in \eqref{eqn:alpha-range}, and
		let $u\in \mathrm{C}^{0}(B_1)$ is a viscosity solution of \eqref{eq:shifted},
		where $\xi\in\mathbb R^n$ is arbitrary. Assume further that
		\[
		\|u\|_{\mathrm{L}^{\infty}(B_1)}\le1
		\quad \text{and} \quad
		\|a\|_{\mathrm{L}^{\infty}(B_1)}+\|f\|_{\mathrm{L}^{\infty}(B_1)}\le\delta,
		\]
		where $\delta$ is the constant from \Cref{lem:one-step-flatness}. Then, there
		exists an affine function
		\[
		\ell_\infty(x)=a_\infty+b_\infty\cdot x
		\]
		such that, for every $0<r<1/2$,
		\begin{equation}\label{eq:normalized-expansion}
			\sup_{B_r}|u-\ell_\infty|
			\le
			C r^{1+\alpha},
		\end{equation}
		where $C$ depends only on $n,\lambda,\Lambda,p^{+}$ and $\alpha$.
	\end{proposition}
	
	\begin{proof}
		Let $\rho\in(0,1/2)$ be the constant given by
		\Cref{lem:one-step-flatness}. For $j=0,1,2,\ldots$, we want to construct affine functions
		\[
		\ell_j(x)=a_j+b_j\cdot x,
		\]
		with $a_0=0$ and $b_0=0$,
		such that, for every $j\ge0$,
		\begin{equation}\label{eq:induction-flatness}
			\sup_{B_{\rho^j}}|u-\ell_j|
			\le
			\rho^{j(1+\alpha)},
		\end{equation}
		and, for every $j\ge0$,
		\begin{equation}\label{eq:coeff-a-increment}
			|a_{j+1}-a_j|
			\le
			C\rho^{j(1+\alpha)},
		\end{equation}
		\begin{equation}\label{eq:coeff-b-increment}
			|b_{j+1}-b_j|
			\le
			C\rho^{j\alpha}.
		\end{equation}
		
		The case $j=0$ in \eqref{eq:induction-flatness} follows from
		$\|u\|_{\mathrm{L}^{\infty}(B_1)}\le1$.
		By induction, assume that $\ell_1$, $\ell_2$, \dots, $\ell_k$ have been constructed and satisfy \eqref{eq:induction-flatness}, \eqref{eq:coeff-a-increment}, and
		\eqref{eq:induction-flatness}. Define the auxiliary function
		\[
		u_k(x)
		:=
		\frac{u(\rho^k x)-\ell_k(\rho^k x)}
		{\rho^{k(1+\alpha)}}.
		\]
		Then, $\|u_k\|_{\mathrm{L}^{\infty}(B_1)}\le 1$ and, moreover, 
		we have
		\[
		\nabla u(\rho^kx)
		=
		b_k+\rho^{k\alpha}\nabla u_k(x).
		\]
		It follows that
		\[
		\nabla u(\rho^kx)+\xi
		=
		\rho^{k\alpha}
		\left(
		\nabla u_k(x)+\xi_k
		\right),
		\quad \text{ where }
		\xi_k:=\frac{\xi+b_k}{\rho^{k\alpha}}.
		\]
		Also,
		\[
		\nabla^2u(\rho^kx)
		=
		\rho^{k(\alpha-1)}\nabla^2u_k(x).
		\]
		Plugging these identities into \eqref{eq:shifted}, we find that $u_k$
		solves
		\begin{equation}\label{eq:rescaled-k}
			|\nabla u_k+\xi_k|^{p_k(x)}
			F_k(\nabla^2u_k)
			+
			A_k(x)|\nabla u_k+\xi_k|^{q_k(x)}
			=
			f_k(x)
			\quad\text{in }B_1,
		\end{equation}
		where $p_k(x):=p(\rho^kx)$, $q_k(x):=q(\rho^kx)$, 
		\[
		F_k(M):=\rho^{k(1-\alpha)}
		F(\rho^{k(\alpha-1)}M),
		\quad
		f_k(x):=
		\rho^{k(1-\alpha(1+p_k(x)))}f(\rho^kx),
		\]
		\[
		\text{and} \quad
		A_k(x):=
		\rho^{k[1-\alpha(p_k(x)-q_k(x)+1)]}a(\rho^kx).
		\]
		Moreover, the operator $F_k$ is uniformly $(\lambda,\Lambda)$-elliptic and satisfies $F_k(0)=0$.
		Furthermore, $0\le q_k(x)\le p_k(x)\le p^{+}$ in $B_1$. Since, by \eqref{eqn:alpha-range}, $1-\alpha(1+p_k(x))\ge0$, we obtain
		\[
		1-\alpha(p_k(x)-q_k(x)+1)\ge0.
		\]
		It follows that
		\[
		\|A_k\|_{\mathrm{L}^{\infty}(B_1)}
		+
		\|f_k\|_{\mathrm{L}^{\infty}(B_1)}
		\le
		\|a\|_{\mathrm{L}^{\infty}(B_1)}
		+
		\|f\|_{\mathrm{L}^{\infty}(B_1)}
		\le
		\delta.
		\]
		Thus, $u_k$ satisfies the hypotheses of \Cref{lem:one-step-flatness} and
		there exists an affine function
		\[
		\bar\ell_k(x)=\bar a_k+\bar b_k\cdot x,
		\] with $|\bar a_k|+|\bar b_k|\le C$ and 
		such that
		\[
		\sup_{B_\rho}|u_k-\bar\ell_k|
		\le
		\rho^{1+\alpha}.
		\]
		We set
		\[
		\ell_{k+1}(x)
		:=
		\ell_k(x)
		+
		\rho^{k(1+\alpha)}\bar a_k
		+
		\rho^{k\alpha}\bar b_k\cdot x.
		\]
		Equivalently,
		\[
		a_{k+1}=a_k+\rho^{k(1+\alpha)}\bar a_k
		\quad \text{and} \quad
		b_{k+1}=b_k+\rho^{k\alpha}\bar b_k.
		\]
		The bounds \eqref{eq:coeff-a-increment} and \eqref{eq:coeff-b-increment}
		follow immediately from $|\bar a_k|+|\bar b_k|\le C$.
		
		Finally, if $x\in B_{\rho^{k+1}}$, then $x=\rho^k z$ with $z\in B_\rho$.
		Therefore,
		\[
		|u(x)-\ell_{k+1}(x)|
		=
		\rho^{k(1+\alpha)}
		|u_k(z)-\bar\ell_k(z)|
		\le
		\rho^{k(1+\alpha)}\rho^{1+\alpha}
		=
		\rho^{(k+1)(1+\alpha)}.
		\]
		The induction is thus complete.
		
		\smallskip
		
		To conclude the proof, i	t remains to pass to the limit. From \eqref{eq:coeff-a-increment} and
		\eqref{eq:coeff-b-increment}, $\{a_j\}_j$ and $\{b_j\}_j$ are Cauchy sequences.
		Let $a_j\to a_\infty$, $b_j\to b_\infty$, and define
		\[
		\ell_\infty(x):=a_\infty+b_\infty\cdot x.
		\]
		For $m>j$,
		\[
		|a_m-a_j|
		\le
		C\sum_{i=j}^{m-1}\rho^{i(1+\alpha)}
		\le
		C\rho^{j(1+\alpha)},
		\]
		and
		\[
		|b_m-b_j|
		\le
		C\sum_{i=j}^{m-1}\rho^{i\alpha}
		\le
		C\rho^{j\alpha}.
		\]
		Letting $m\to\infty$, we obtain
		\[
		|a_\infty-a_j|\le C\rho^{j(1+\alpha)}
		\quad \text{and} \quad
		|b_\infty-b_j|\le C\rho^{j\alpha}.
		\]
		Consequently,
		\[
		\sup_{B_{\rho^j}}|\ell_\infty-\ell_j|
		\le
		|a_\infty-a_j|
		+
		\rho^j|b_\infty-b_j|
		\le
		C\rho^{j(1+\alpha)}.
		\]
		Combining this with \eqref{eq:induction-flatness}, we get
		\[
		\sup_{B_{\rho^j}}|u-\ell_\infty|
		\le
		C\rho^{j(1+\alpha)}.
		\]
		Given $0<r<1/2$, choose $j\ge0$ such that $\rho^{j+1}<r\le \rho^j$. Then,
		\[
		\rho^{j(1+\alpha)}
		\le
		\rho^{-(1+\alpha)}r^{1+\alpha}.
		\]
		Therefore, after increasing $C$ if necessary, we have
		\eqref{eq:normalized-expansion}.
	\end{proof}

	We are now ready to prove the first of our main results.
	
	\begin{proof}[Proof of \Cref{teo_c1alpha}]
		Fix $\alpha$ as in the assumption \eqref{eqn:alpha-range} and let $\Omega'\subset\subset\Omega$. We prove the estimate in $\Omega'$.
		
		First, we reduce the problem to the normalized regime, as in \Cref{subsec_small}. Let $R:=\frac12\operatorname{dist}(\Omega',\partial\Omega)$,
		fix $y\in\Omega'$, and define, for $x\in B_1$,
		\[
		v(x):=\frac{u(y+Rx)}{K},
		\]
		where $K\ge1$ is to be chosen below. Set
		\[
		\bar p(x):=p(y+Rx)
		\quad\text{and}\quad
		\bar q(x):=q(y+Rx).
		\]
		A direct computation gives
		\[
		|\nabla v|^{\bar p(x)}\bar F(\nabla^2v)
		+
		\bar a(x)|\nabla v|^{\bar q(x)}
		=
		\bar f(x)
		\qquad\text{in }B_1,
		\]
		where
		\[
		\bar F(M)
		:=
		\frac{R^2}{K}
		F\left(\frac{K}{R^2}M\right),
		\]
		\[
		\bar a(x)
		:=
		R^{\bar p(x)-\bar q(x)+2}
		K^{\bar q(x)-\bar p(x)-1}
		a(y+Rx),
		\]
		and
		\[
		\bar f(x)
		:=
		\frac{R^{\bar p(x)+2}}
		{K^{\bar p(x)+1}}
		f(y+Rx).
		\]
		The operator $\bar F$ is uniformly $(\lambda,\Lambda)$-elliptic and satisfies
		$\bar F(0)=0$.
		Let
		\[
		p^{+}:=\sup_\Omega p(x),
		\quad
		p_-:=\inf_\Omega p(x),
		\quad\text{and}\quad
		d_-:=\inf_\Omega(p(x)-q(x)).
		\]
		By assumption, $p_-\ge0$ and $d_-\ge0$. Choose
		\[
		K
		:=
		2\left(
		1+\|u\|_{\mathrm{L}^{\infty}(\Omega)}
		+
		\delta^{-\frac{1}{1+d_-}}
		\|a\|_{\mathrm{L}^{\infty}(\Omega)}^{\frac{1}{1+d_-}}
		+
		\delta^{-\frac{1}{1+p_-}}
		\|f\|_{\mathrm{L}^{\infty}(\Omega)}^{\frac{1}{1+p_-}}
		\right),
		\]
		where $\delta > 0$ is the smallness constant in
		\Cref{lem:one-step-flatness}. Since $R$ is fixed in terms of
		$\operatorname{dist}(\Omega',\partial\Omega)$, increasing the universal constant
		if necessary, this choice guarantees
		\[
		\|v\|_{\mathrm{L}^{\infty}(B_1)}\le1
		\]
		and
		\[
		\|\bar a\|_{\mathrm{L}^{\infty}(B_1)}
		+
		\|\bar f\|_{\mathrm{L}^{\infty}(B_1)}
		\le
		\delta.
		\]
		Then, applied with $\xi=0$, \Cref{prop:normalized-c1alpha} yields an affine
		function
		\[
		\ell_y(x)=\alpha_y+\beta_y\cdot x
		\]
		such that
		\[
		\sup_{B_r}|v-\ell_y|
		\le
		Cr^{1+\alpha},
		\]
		for every $0<r<1/2$, with $C$ depending only on
		$n,\lambda,\Lambda,p^{+}$ and $\alpha$. Rescaling back, we obtain an affine function
		\[
		L_y(z)=u(y)+B_y\cdot(z-y)
		\]
		such that
		\[
		\sup_{B_s(y)}|u-L_y|
		\le
		C K R^{-(1+\alpha)}s^{1+\alpha}
		\]
		for every $0<s<R/2$. Since $R$ depends only on
		$\operatorname{dist}(\Omega',\partial\Omega)$, the constant above is bounded by
		\[
		C
		\left(
		1+\|u\|_{\mathrm{L}^{\infty}(\Omega)}
		+
		\|a\|_{\mathrm{L}^{\infty}(\Omega)}^{\frac{1}{1+d_-}}
		+
		\|f\|_{\mathrm{L}^{\infty}(\Omega)}^{\frac{1}{1+p_-}}
		\right),
		\]
		where now $C$ depends on
		$n,\lambda,\Lambda,\alpha,p^{+},\operatorname{dist}(\Omega',\partial\Omega)$. Since the constants are uniform in $y\in\Omega'$, it is now standard to see that $u\in \mathrm{C}^{1,\alpha}(\Omega')$ and
		\[
		[\nabla u]_{\mathrm{C}^{0,\alpha}(\Omega')}
		\le
		C
		\left(
		1+\|u\|_{\mathrm{L}^{\infty}(\Omega)}
		+
		\|a\|_{\mathrm{L}^{\infty}(\Omega)}^{\frac{1}{1+\inf_\Omega(p-q)}}
		+
		\|f\|_{\mathrm{L}^{\infty}(\Omega)}^{\frac{1}{1+\inf_\Omega p}}
		\right).
		\]
		This proves the theorem.
	\end{proof}

	\section{Improved pointwise regularity}\label{sec:improved-pointwise}
	
	In this section, we prove 
	\Cref{Thm2}. After subtracting the constant $u(0)$, we assume throughout this section that
	\[
	u(0)=0.
	\]
	The key additional observation is that, if $|\nabla u(0)|$ is small, then
	the compactness limit is not merely a solution of a homogeneous uniformly
	elliptic equation, but a homogeneous profile $h$ satisfying
	\[
	h(0)=0
	\quad\text{and}\quad
	\nabla h(0)=0.
	\]
	This extra information improves the decay at the origin and allows us to
	iterate an oscillation reduction.
	For convenience, we introduce the notation
	\[
	\mathcal D(h):=
	\{x\in B_1;\ h(x)=0\ \text{and}\ \nabla h(x)=0\}.
	\]
	
	\begin{lemma}[Improved approximation lemma]\label{lem:improved-approximation}
		Assume \ref{hipoteseF} -- \ref{hip2}, and let $u\in \mathrm{C}^{0}(B_1)$ be a normalized viscosity solution of \eqref{eq principal} 
		with $u(0)=0$. Given $\varepsilon>0$, there exists $\delta=\delta(n,\lambda,\Lambda,p^{+},\varepsilon)>0$ such that, if
		\[
		|\nabla u(0)|
		+
		\|a\|_{\mathrm{L}^{\infty}(B_1)}
		+
		\|f\|_{\mathrm{L}^{\infty}(B_1)}
		\le \delta;
		\]
		then, there exist a uniformly $(\lambda,\Lambda)$-elliptic operator
		$\mathcal F:\mathcal S(n)\to\mathbb R$, with $\mathcal F(0)=0$, and a viscosity
		solution $h$ of $\mathcal F(\nabla^2h)=0$ in $B_{3/4}$,
		such that $0\in\mathcal D(h)$ and
		\[
		\|u-h\|_{\mathrm{L}^{\infty}(B_{1/2})}\le \varepsilon.
		\]
	\end{lemma}
	
	\begin{proof}
		We argue by contradiction. Let us assume that the lemma is false; then, there exist $\varepsilon_0>0$ and sequences of $\{a_j\},\{p_j\},\{q_j\},\{f_j\},\{u_j\},\{F_j\}$ satisfying 
		\begin{enumerate}[label=$(\roman*)$]
			\item $u_j\in \mathrm{C}^{0}(B_1)$, $u_j(0)=0$ and $\|u_j\|_{\mathrm{L}^{\infty}(B_1)}\leq 1$;
			\smallskip
			\item $F_j$ is uniformly $(\lambda,\Lambda)$-elliptic;
			\item $0\leq q_j(x)\leq p_j(x) \leq \sup_{B_1}p(x)$ and $p_j,q_j\in \mathrm{C}^{0}(B_1)$;
			\smallskip
			\smallskip
			\item  $|\nabla u_j(0)|+\|f_j\|_{\mathrm{L}^{\infty}(B_1)}+ \|a_j\|_{\mathrm{L}^{\infty}(B_1)}\leq \frac{1}{j}$ and $f_j,a_j\in \mathrm{C}^{0}(B_1)$;
			\smallskip
		\end{enumerate}
		and such that
		\begin{equation}
			|\nabla u_j|^{p_j(x)}F_j(\nabla^2u_j)+a_j(x)|\nabla u_j|^{q_j(x)}=f_j(x)  \quad \text{in } B_1,
		\end{equation}
		but
		\begin{equation}\label{admissiblefunc}
			\|u_j-h\|_{\mathrm{L}^{\infty}(B_{1/2})}>\varepsilon_0,
		\end{equation}
		for every viscosity solution $h$, with $0\in\mathcal D(h)$, of a homogeneous uniformly elliptic
		equation $\mathcal F(\nabla^2h)=0$ in $B_{3/4}$,
		with $\mathcal F(0)=0$, and $\mathcal F$ uniformly $(\lambda,\Lambda)$-elliptic.
		
		By \Cref{teo_c1alpha}, the sequence $\{u_j\}_{j\in \mathbb{N}}$ is uniformly bounded in
		$\mathrm{C}^{1,\alpha}_{\rm loc}(B_1)$ for some universal $\alpha\in(0,1)$. Hence, up to
		a subsequence, both
		\[
		u_j\to u_\infty
		\quad \text{and}
		\quad
		\nabla u_j\to \nabla u_\infty
		\] locally uniformly in $B_1$.
		In particular, $u_\infty(0)=0$ and $\nabla u_\infty(0)=0$.

		Since the operators $F_j$ are uniformly elliptic and normalized by $F_j(0)=0$,
		we may also assume, up to a subsequence, that
		\[
		F_j\to F_\infty
		\quad\text{locally uniformly in }\mathcal S(n),
		\]
		where $F_\infty$ is uniformly $(\lambda,\Lambda)$-elliptic and
		$F_\infty(0)=0$.
		By stability, we conclude $u_\infty$ is a viscosity
		solution of
		\[
		F_\infty(\nabla^2u_\infty)=0
		\quad\text{in }B_{3/4}.
		\]
		Moreover, $u_\infty(0)=0$ and $\nabla u_\infty(0)=0$, and therefore $0\in\mathcal D(u_\infty)$. Thus, $u_\infty$ is an admissible
		comparison function in \eqref{admissiblefunc}. Taking
		$h=u_\infty$ gives
		\[
		\|u_j-u_\infty\|_{\mathrm{L}^{\infty}(B_{1/2})}>\varepsilon_0
		\]
		for every $j$, contradicting the locally uniform convergence of $u_j$ to
		$u_\infty$.
	\end{proof}

	The next lemma is the one-step oscillation reduction at points where the gradient is small.
	
	\begin{lemma}[Oscillation reduction]\label{lem:oscillation-reduction}
		Assume \ref{hipoteseF} -- \ref{hip2}, and let $0<\gamma<\alpha_0$. Then, there exist constants
		\[
		\rho_0\in(0,1/2)
		\quad\text{and}\quad
		\tilde\delta\in(0,1),
		\]
		depending only on $n,\lambda,\Lambda,p^{+}$ and $\gamma$, such that the following
		holds: if $u\in \mathrm{C}^{0}(B_1)$ is a normalized viscosity solution of \eqref{eq principal},
		with $u(0)=0$, such that
		\[
		|\nabla u(0)|
		+
		\|a\|_{\mathrm{L}^{\infty}(B_1)}
		+
		\|f\|_{\mathrm{L}^{\infty}(B_1)}
		\le \tilde\delta;
		\]
		then,
		\[
		\sup_{B_{\rho_{0}}}|u|
		\le
		\rho_0^{1+\gamma}.
		\]
	\end{lemma}
	
	\begin{proof}
		Let $\varepsilon>0$ to be chosen later. By
		\Cref{lem:improved-approximation}, provided $\tilde\delta$ is small enough,
		there exist a uniformly $(\lambda,\Lambda)$-elliptic operator $\mathcal F$ and a
		solution $h$ of $\mathcal F(\nabla^2h)=0$ in $B_{3/4}$,
		such that $0\in\mathcal D(h)$ and
		\[
		\|u-h\|_{\mathrm{L}^{\infty}(B_{1/2})}\le\varepsilon.
		\]
		Since $h(0)=0$ and $\nabla h(0)=0$, the universal $\mathrm{C}^{1,\alpha_0}$ estimate gives
		\[
		\sup_{B_{\rho_0}}|h|
		\le
		C\rho_0^{1+\alpha_0},
		\]
		where $C=C(n,\lambda,\Lambda)$. Choose first $\rho_0\in(0,1/2)$ sufficiently small so that
		\[
		C\rho_0^{\alpha_0-\gamma}\le \frac12
		\quad
		\text{or, equivalently,}
		\quad
		C\rho_0^{1+\alpha_0}
		\le
		\frac12\rho_0^{1+\gamma}.
		\]
		Next, choose $\varepsilon:=\frac12\rho_0^{1+\gamma}$. Then,
		\[
		\sup_{B_{\rho_0}}|u|
		\le
		\sup_{B_{\rho_0}}|u-h|
		+
		\sup_{B_{\rho_0}}|h|
		\le
		\frac12\rho_0^{1+\gamma}
		+
		\frac12\rho_0^{1+\gamma}
		=
		\rho_0^{1+\gamma}.
		\]
		This proves the lemma.
	\end{proof}
	
	Next, we iterate the previous oscillation reduction. In this proposition, we work
	under normalized growth assumptions on the coefficients. The passage from the
	general constants $K_1,K_2$ in \eqref{source1} to this normalized setting will be
	made in the proof of \Cref{Thm2}.
	
	\begin{proposition}[Growth estimate in the small-gradient regime]
		\label{prop:small-gradient-growth}
		Assume \ref{hipoteseF} -- \ref{hip2}, and let $\rho_0$ and $\tilde\delta$ be as in
		\Cref{lem:oscillation-reduction}. Assume further that $u\in \mathrm{C}^{0}(B_1)$ is a normalized
		viscosity solution of \eqref{eq principal}, with $u(0)=0$, and the normalized growth conditions
		\begin{equation}\label{eq:normalized-growth}
			|f(x)|\le \frac{\tilde\delta}{2}|x|^{\theta_1}
			\quad\text{and}\quad
			|a(x)|\le \frac{\tilde\delta}{2}|x|^{\theta_2}
			\quad\text{in }B_1.
		\end{equation}
		Let $\gamma$ as in \eqref{eqn:gamma-thm2}.
		Then, there exists $C>0$, depending only on
		$n,\lambda,\Lambda,p^{+},\gamma,\rho_0$ and $\tilde\delta$, such that, if
		$0<r\le \rho_0$ and
		\[
		|\nabla u(0)|\le \tilde\delta r^\gamma;
		\]
		then,
		\begin{equation}\label{eq:small-gradient-growth}
			\sup_{B_r}|u|
			\le
			C r^{1+\gamma}.
		\end{equation}
	\end{proposition}
	
	\begin{proof}
		We first prove the discrete estimate
		\begin{equation}\label{eq:discrete-small-gradient}
			|\nabla u(0)|\le \tilde\delta \rho_0^{k\gamma}
			\quad\implies\quad
			\sup_{B_{\rho_0^k}}|u|
			\le
			\rho_0^{k(1+\gamma)}
		\end{equation}
		for every integer $k\ge1$.
		
		For $k=1$, this is exactly \Cref{lem:oscillation-reduction}, since $|\nabla u(0)|\le \tilde\delta \rho_0^\gamma\le\tilde\delta$ and, by \eqref{eq:normalized-growth},
		\[
		\|a\|_{\mathrm{L}^{\infty}(B_1)}+\|f\|_{\mathrm{L}^{\infty}(B_1)}
		\le \tilde\delta.
		\]
		By induction,
		assume now that \eqref{eq:discrete-small-gradient} holds for some $k\ge1$, and
		suppose
		\[
		|\nabla u(0)|\le \tilde\delta \rho_0^{(k+1)\gamma}.
		\]
		Consider the auxiliary rescaled function
		\[
		v_k(x):=
		\frac{u(\rho_0^kx)}
		{\rho_0^{k(1+\gamma)}}.
		\]
		By the induction hypothesis, $\|v_k\|_{\mathrm{L}^{\infty}(B_1)}\le1$, and $v_k(0)=0$. Moreover,
		\[
		\nabla v_k(0)
		=
		\rho_0^{-k\gamma}\nabla u(0)
		\quad
		\text{and so}
		\quad
		|\nabla v_k(0)|
		\le
		\tilde\delta \rho_0^\gamma
		\le
		\tilde\delta.
		\]
		Moreover, a direct computation shows that $v_k$ solves
		\[
		|\nabla v_k|^{p_k(x)}F_k(\nabla^2v_k)
		+
		a_k(x)|\nabla v_k|^{q_k(x)}
		=
		f_k(x)
		\qquad\text{in }B_1,
		\]
		where $p_k(x):=p(\rho_0^kx)$, $q_k(x):=q(\rho_0^kx)$,
		\[
		F_k(M):=
		\rho_0^{k(1-\gamma)}
		F\left(\rho_0^{-k(1-\gamma)}M\right),
		\]
		\[
		f_k(x):=
		\rho_0^{k[1-\gamma(1+p_k(x))]}f(\rho_0^kx),
		\]
		and
		\[
		a_k(x):=
		\rho_0^{k[1-\gamma(p_k(x)-q_k(x)+1)]}a(\rho_0^kx).
		\]
		The operator $F_k$ is uniformly $(\lambda,\Lambda)$-elliptic and 
		$F_k(0)=0$.
		By \eqref{eq:normalized-growth}, we estimate, for $x\in B_1$,
		\[
		|f_k(x)|
		\le
		\frac{\tilde\delta}{2}
		\rho_0^{k[1-\gamma(1+p_k(x))+\theta_1]}.
		\]
		By the assumptions on $\gamma$, $1-\gamma(1+p_k(x))+\theta_1\ge0$ and so
		\[
		\|f_k\|_{\mathrm{L}^{\infty}(B_1)}\le \frac{\tilde\delta}{2}.
		\]
		Similarly,
		\[
		\gamma\le
		\frac{1+\theta_2}{1+\sup_{B_1}(p-q)}
		\implies
		\|a_k\|_{\mathrm{L}^{\infty}(B_1)}\le \frac{\tilde\delta}{2}.
		\]
		Hence,
		\[
		|\nabla v_k(0)|
		+
		\|a_k\|_{\mathrm{L}^{\infty}(B_1)}
		+
		\|f_k\|_{\mathrm{L}^{\infty}(B_1)}
		\le
		\tilde\delta.
		\]
		Applying \Cref{lem:oscillation-reduction} to $v_k$, we obtain
		\[
		\sup_{B_{\rho_0}}|v_k|
		\le
		\rho_0^{1+\gamma}.
		\]
		Rescaling back gives
		\[
		\sup_{B_{\rho_0^{k+1}}}|u|
		\le
		\rho_0^{(k+1)(1+\gamma)}
		\]
		and the discrete estimate is proved.
		
		Now, let $0<r\le\rho_0$, assume $|\nabla u(0)|\le \tilde\delta r^\gamma$, and choose $k\ge1$ such that
		\[
		\rho_0^{k+1}<r\le \rho_0^k.
		\]
		Then,
		\[
		|\nabla u(0)|
		\le
		\tilde\delta r^\gamma
		\le
		\tilde\delta \rho_0^{k\gamma}.
		\]
		By the discrete estimate, and
		since $B_r\subset B_{\rho_0^k}$,
		\[
		\sup_{B_r}|u|
		\le
		\rho_0^{k(1+\gamma)}.
		\]
		Finally, from $\rho_0^{k+1}<r$, we have $\rho_0^k<\rho_0^{-1}r$ and, therefore, $\rho_0^{k(1+\gamma)}
		\le
		\rho_0^{-(1+\gamma)}r^{1+\gamma}$. This proves \eqref{eq:small-gradient-growth}.
	\end{proof}

	We are now ready to prove \Cref{Thm2}. The proof is first carried out in
	the normalized setting; then, the standard scaling argument from \Cref{subsec_small} gives the general form.
	
	\begin{proof}[Proof of \Cref{Thm2}]
		First, we prove the result under the normalized assumptions
		\[
		\|u\|_{\mathrm{L}^{\infty}(B_1)}\le1,
		\quad
		|f(x)|\le \frac{\tilde\delta}{2}|x|^{\theta_1},
		\quad \text{and} \quad 
		|a(x)|\le \frac{\tilde\delta}{2}|x|^{\theta_2}
		\quad\text{in }B_1.
		\]
		Fix $\gamma$ as in \eqref{eqn:gamma-thm2}
		and set $P:=\nabla u(0)$. We prove that, for every $0<r\le\rho_0$,
		\begin{equation}\label{eq:normalized-thm2-estimate}
			\sup_{B_r}|u(x)-P\cdot x|
			\le
			C r^{1+\gamma}.
		\end{equation}
		If $P=0$, \Cref{prop:small-gradient-growth} applies at every
		$0<r\le\rho_0$, and \eqref{eq:normalized-thm2-estimate} follows immediately.
		Then, suppose $P\neq0$ and define the transition scale
		\[
		\varrho:=
		\left(\frac{|P|}{\tilde\delta}\right)^{1/\gamma}.
		\]
		
		\smallskip
		
		\noindent\textit{Case 1:} $\varrho\le\rho_0$.
		If $\rho\in[\varrho,\rho_0]$, then, since $|P|=\tilde\delta\varrho^\gamma\le \tilde\delta\rho^\gamma$,  \Cref{prop:small-gradient-growth} gives
		\[
		\sup_{B_\rho}|u|
		\le
		C\rho^{1+\gamma}.
		\] and, hence,
		\[
		\sup_{B_\rho}|u(x)-P\cdot x|
		\le
		\sup_{B_\rho}|u|
		+
		|P|\rho
		\le
		C\rho^{1+\gamma}
		+
		\tilde\delta\rho^{1+\gamma}
		\le
		C\rho^{1+\gamma}.
		\]
		Thus, it remains to address the case $0<\rho<\varrho$. Define
		\[
		v(x):=
		\frac{u(\varrho x)}
		{\varrho^{1+\gamma}}.
		\]
		Then, $v(0)=0$ and $|\nabla v(0)|
		=
		\tilde\delta$.
		Moreover, applying the estimate already proved at scale $\varrho$, we get
		\[
		\|v\|_{\mathrm{L}^{\infty}(B_1)}\le C.
		\]
		The function $v$ solves
		\[
		|\nabla v|^{\bar p(x)}\bar F(\nabla^2v)
		+
		\bar a(x)|\nabla v|^{\bar q(x)}
		=
		\bar f(x),
		\quad\text{in }B_1,
		\]
		where
		\[
		\bar p(x):=p(\varrho x),
		\quad
		\bar q(x):=q(\varrho x),
		\quad 
		\bar F(M):=
		\varrho^{1-\gamma}F(\varrho^{\gamma-1}M),\\[0.5em]
		\]
		\[
		\bar f(x):=
		\varrho^{1-\gamma(1+\bar p(x))}f(\varrho x),
		\quad
		\text{and}
		\quad
		\bar a(x):=
		\varrho^{1-\gamma(\bar p(x)-\bar q(x)+1)}a(\varrho x).\\[0.3em]
		\]
		By the growth assumptions and the choice of $\gamma$,
		\[
		\|\bar f\|_{\mathrm{L}^{\infty}(B_1)}
		+
		\|\bar a\|_{\mathrm{L}^{\infty}(B_1)}
		\le C.
		\]
		Hence, the $\mathrm{C}^{1,\alpha}$ estimate from \Cref{teo_c1alpha} gives
		\[
		[\nabla v]_{\mathrm{C}^{0,\alpha}(B_{1/2})}\le C
		\]
		for some universal $\alpha\in(0,1)$. Therefore, choosing a universal
		$r_*\in(0,1/2)$ sufficiently small, we have
		\[
		|\nabla v(x)-\nabla v(0)|
		\le
		\frac{\tilde\delta}{2},
		\quad\text{for }x\in B_{r_*}.
		\]
		Consequently,
		\[
		|\nabla v(x)|\ge \frac{\tilde\delta}{2},
		\quad\text{in }B_{r_*}.
		\]
		In $B_{r_*}$ the equation is uniformly elliptic. Indeed, dividing by
		$|\nabla v|^{\bar p(x)}$, we obtain
		\[
		\bar F(\nabla^2v)
		=
		|\nabla v|^{-\bar p(x)}\bar f(x)
		-
		\bar a(x)|\nabla v|^{\bar q(x)-\bar p(x)}
		\quad\text{in }B_{r_*},
		\]
		and the right-hand side is bounded by a universal constant. By the standard
		interior $\mathrm{C}^{1,\alpha_0}$ estimate for uniformly elliptic equations with bounded
		right-hand side, and since $\gamma<\alpha_0$, we obtain
		\[
		\sup_{B_\sigma}
		|v(x)-\nabla v(0)\cdot x|
		\le
		C\sigma^{1+\gamma},
		\]
		for every $0<\sigma\le r_*/2$.
		Rescaling back gives
		\[
		\sup_{B_\rho}
		|u(x)-P\cdot x|
		\le
		C\rho^{1+\gamma},
		\quad
		\text{for every }
		0<\rho\le \frac{r_*}{2}\varrho.
		\]
		For the intermediate scales
		\[
		\frac{r_*}{2}\varrho<\rho<\varrho,
		\]
		we use the estimate already proved at scale $\varrho$:
		\[
		\sup_{B_\rho}|u(x)-P\cdot x|
		\le
		\sup_{B_\varrho}|u(x)-P\cdot x|
		\le
		C\varrho^{1+\gamma}.
		\]
		Since $\rho>(r_*/2)\varrho$, we have
		\[
		\varrho^{1+\gamma}
		\le
		\left(\frac2{r_*}\right)^{1+\gamma}\rho^{1+\gamma}.
		\]
		Thus,
		\[
		\sup_{B_\rho}|u(x)-P\cdot x|
		\le
		C\rho^{1+\gamma}
		\]
		also in the intermediate range. This completes the proof of
		\eqref{eq:normalized-thm2-estimate} when $\varrho\le\rho_0$.
		
		\smallskip
		
		\noindent\textit{Case 2:} $\varrho>\rho_0$.
		In this case $|P|>\tilde\delta\rho_0^\gamma$. By \Cref{teo_c1alpha}, $\nabla u$ is Hölder continuous in $B_{1/2}$ with a
		universal bound in the normalized regime. Hence, there exists a universal
		$r_0\in(0,\rho_0)$ such that
		\[
		|\nabla u(x)-P|
		\le
		\frac12\tilde\delta\rho_0^\gamma,
		\quad\text{for }x\in B_{r_0}.
		\]
		Then,
		\[
		|\nabla u(x)|\ge
		\frac12\tilde\delta\rho_0^\gamma,
		\quad\text{in }B_{r_0}.
		\]
		Thus, the equation is uniformly elliptic in $B_{r_0}$ after division by
		$|\nabla u|^{p(x)}$, with bounded right-hand side. The standard interior
		$\mathrm{C}^{1,\alpha_0}$ estimate gives
		\[
		\sup_{B_r}|u(x)-P\cdot x|
		\le
		Cr^{1+\gamma},
		\quad\text{for }0<r\le r_0/2.
		\]
		For $r_0/2<r\le \rho_0$, the estimate follows from the boundedness of $u$ and
		the universal bound for $|P|$, after increasing possibly $C$. Hence,
		\[
		\sup_{B_r}|u(x)-P\cdot x|
		\le
		Cr^{1+\gamma},
		\quad\text{for }0<r\le \rho_0.
		\]
		
		\smallskip
		
		Combining the two cases proves the normalized estimate
		\[
		\sup_{B_r}|u(x)-\nabla u(0)\cdot x|
		\le
		Cr^{1+\gamma},
		\quad\text{for }0<r\le\rho_0.
		\]
		For $\rho_0<r\le1/2$, the same estimate follows from the boundedness of $u$ and
		$\nabla u(0)$, again after increasing $C$.
		
		Finally, we can now remove the normalization assumption by the standard scaling argument from \Cref{subsec_small}. This proves \Cref{Thm2}
		%
%		We now remove the normalization. For the original solution, define
%		\[
%		\widetilde u(x):=
%		\frac{u(\tau x)-u(0)}
%		{K},
%		\]
%		where $K\ge1$ and $\tau\in(0,1)$ are chosen so that $\|\widetilde u\|_{\mathrm{L}^{\infty}(B_1)}\le1$, and the rescaled coefficients satisfy
%		\[
%		|\widetilde f(x)|\le\frac{\tilde\delta}{2}|x|^{\theta_1}
%		\quad\text{and}\quad
%		|\widetilde a(x)|\le\frac{\tilde\delta}{2}|x|^{\theta_2},
%		\quad\text{in }B_1.
%		\]
%		This is achieved by the same scaling computation used in the preliminary
%		smallness regime, with $K$ depending on $\|u\|_{\mathrm{L}^{\infty}(B_1)}$ and with $\tau$
%		depending on the structural constants, $K_1,K_2,\theta_1,\theta_2$, and
%		$\tilde\delta$.
%		%
%		Applying the normalized estimate to $\widetilde u$ and rescaling back, we obtain
%		\[
%		\sup_{B_r}
%		|u(x)-u(0)-\nabla u(0)\cdot x|
%		\le
%		C r^{1+\gamma},
%		\]
%		for every $0<r\le1/2$, where $C$ depends on
%		$\|u\|_{\mathrm{L}^{\infty}(B_1)}$ and on the universal parameters $n,\lambda,\Lambda,p^{+},\sup_{B_1}(p-q),
%		\theta_1,\theta_2,K_1,K_2,\gamma$. This proves \Cref{Thm2}.
	\end{proof}

	\section{Pointwise Schauder estimate at extremal points}
	\label{sec:schauder-extremal}

	In this final section, we prove \Cref{SchauderEstimate}, which gives a pointwise Schauder-type estimate at local minimum
	points. The argument follows the same compactness-and-scaling philosophy used in
	the previous sections, but now the limiting homogeneous profile is forced to be
	constant by the strong maximum principle.
	
	We begin with the flatness lemma around extremal points that drives the iteration in the proof of \Cref{SchauderEstimate}.
	
	\begin{lemma}[Flatness at an extremal point]\label{lemaC2}
		Assume \ref{hipoteseF} -- \ref{hip2}, and let $u\in \mathrm{C}^{0}(B_1)$ be a normalized viscosity solution of \eqref{eq:theta-schauder-assumptions}. Assume further that $0$ is a local minimum of $u$. Then, for every
		$\varepsilon>0$, there exists a constant $\delta_0=\delta_0(n,\lambda,\Lambda,p^{+},\varepsilon)>0$ such that, if
		\[
		\|a\|_{\mathrm{L}^{\infty}(B_1)}
		+
		\|f\|_{\mathrm{L}^{\infty}(B_1)}
		\le
		\delta_0;
		\]
		then,
		\[
		\sup_{B_{1/2}}(u-u(0))
		\le
		\varepsilon.
		\]
	\end{lemma}
	
	\begin{proof}
		By contradiction, let us assume that there exist a constant $\varepsilon_0>0$, and sequences $\{a_j\}$, $\{p_j\}$, $\{q_j\}$, $\{f_j\}$, $\{u_j\}$, and $\{F_j\}$ satisfying 
		\begin{enumerate}[label=$(\roman*)$]
			\item $u_j\in \mathrm{C}^{0}(B_1)$, $\|u_j\|_{\mathrm{L}^{\infty}(B_1)}\leq 1$, and $0$ local minimum of $u_{j}$;
			\smallskip
			\item $F_j$ is uniformly $(\lambda,\Lambda)$-elliptic and $F_{j}(0)=0$;
			\smallskip
			\item $f_j,a_j\in \mathrm{C}^{0}(B_1)$ with  $\|a_j\|_{\mathrm{L}^{\infty}(B_1)}+\|f_j\|_{\mathrm{L}^{\infty}(B_1)}\leq 1/j$;
			\smallskip
			\item $p_j,q_j\in \mathrm{C}^{0}(B_1)$ with $0\leq q_j(x)\leq p_j(x)\leq p^{+}$;
		\end{enumerate}
		and such that $u_{j}$ is a viscosity solution of
		\begin{equation}
			|\nabla u_j|^{p_j(x)}F_j(\nabla^2u_j)+a_j(x)|\nabla u_j|^{q_j(x)}=f_j(x)  \quad \text{in } B_1,
		\end{equation}
		but
		\begin{equation}\label{minimum}
			\sup_{B_{1/2}} \{u_j-u_j(0)\} >\varepsilon_0.
		\end{equation}
		We argue as in \Cref{lem2-4} and there exist $u_{\infty}\in \mathrm{C}^{0}(B_1)$ and $F_{\infty}$ uniformly $(\lambda,\Lambda)$-elliptic, with $F_{\infty}(0) = 0$, such that, up to subsequences, $u_j\rightarrow u_{\infty}$ and $F_j\rightarrow F_{\infty}$ locally uniformily. Moreover, by stability of viscosity solutions,
		$$F_{\infty}(\nabla^2 u_{\infty})=0,\quad\text{in } B_{3/4},$$
		in the viscosity sense. From the uniform convergence of $u_j$,  $0$ is a local minimum of $u_{\infty}$. Thus, by the strong maximum principle, $u_{\infty}$ is constant in $B_{3/4}$, which contradicts \eqref{minimum}.
	\end{proof}
	
	We now prove the theorem by iterating the preceding flatness lemma. As before, the proof is
	carried out in a normalized regime; the general case following by the scaling argument from \Cref{subsec_small}.
	
	\begin{proof}[Proof of \Cref{SchauderEstimate}]
		Since $0$ is a local minimum of $u$, after possibly reducing $r_0$ we may work in
		a ball where $u(x)\ge u(0)$. Subtracting the constant $u(0)$, we assume in the normalized part of the proof
		that $u(0)=0$.
		
		\smallskip
		
		\noindent\textit{Step 1:} normalized estimate.
		Assume first that $\|u\|_{\mathrm{L}^{\infty}(B_1)}\le1$ and that the growth assumptions are normalized as
		\begin{equation}\label{eq:normalized-schauder-growth}
			|f(x)|\le \frac{\delta}{2}\,|x|^{\theta_1}
			\quad \text{and} \quad
			|a(x)|\le \frac{\delta}{2}\,|x|^{\theta_2}
			\quad\text{in }B_1,
		\end{equation}
		where $\delta>0$ is to be chosen below.
		Let $\varepsilon:=2^{-(2+\tilde\alpha)}$ and choose
		\[
		\delta\le \delta_0(n,\lambda,\Lambda,p^{+},\varepsilon),
		\]
		where $\delta_0$ is the smallness constant in \Cref{lemaC2}. Then, the flatness
		lemma gives
		\begin{equation}\label{eq:first-flatness-schauder}
			\sup_{B_{1/2}}u
			\le
			2^{-(2+\tilde\alpha)}.
		\end{equation}
		We claim that, for every integer $k\ge1$, there holds
		\begin{equation}\label{eq:dyadic-schauder-estimate}
			\sup_{B_{2^{-k}}}u
			\le
			2^{-k(2+\tilde\alpha)}.
		\end{equation}
		The case $k=1$ is \eqref{eq:first-flatness-schauder}.
		Assume, by induction, \eqref{eq:dyadic-schauder-estimate} holds up to $k\ge1$, and define
		\[
		v_k(x)
		:=
		2^{k(2+\tilde\alpha)}
		u(2^{-k}x),
		\quad x\in B_1.
		\]
		Then, $v_k(0)=0$, $0$ is a local minimum of $v_k$, and by the induction hypothesis
		\[
		\|v_k\|_{\mathrm{L}^{\infty}(B_1)}\le1.
		\]
		We compute the equation satisfied by $v_k$. If
		\[
		p_k(x):=p(2^{-k}x)
		\quad\text{and}\quad
		q_k(x):=q(2^{-k}x),
		\]
		then $v_k$ is a viscosity solution of
		\[
		|\nabla v_k|^{p_k(x)}F_k(\nabla^2v_k)
		+
		a_k(x)|\nabla v_k|^{q_k(x)}
		=
		f_k(x)
		\qquad\text{in }B_1,
		\]
		where
		\[
		F_k(M):=
		2^{k\tilde\alpha}F(2^{-k\tilde\alpha}M),\\[0.5em]
		\]
		\[
		f_k(x):=
		2^{k[(1+\tilde\alpha)p_k(x)+\tilde\alpha]}
		f(2^{-k}x),
		\
		\text{and}\\[0.5em]
		\]
		\[
		a_k(x):=
		2^{k[(1+\tilde\alpha)(p_k(x)-q_k(x))+\tilde\alpha]}
		a(2^{-k}x).\\[0.5em]
		\]
		The operator $F_k$ is uniformly $(\lambda,\Lambda)$-elliptic and 
		$F_k(0)=0$.
		
		We now check that the rescaled coefficients remain small. Observe
		\[
		\tilde\alpha
		\le
		\frac{\theta_1-p^{+}}{1+p^{+}}
		\implies
		(1+\tilde\alpha)p_k(x)+\tilde\alpha-\theta_1\le0,
		\] and, hence, from
		\eqref{eq:normalized-schauder-growth},
		\[
		\|f_k\|_{\mathrm{L}^{\infty}(B_1)}\le \frac{\delta}{2}.
		\]
		Similarly, the choice of $\bar \alpha$ ensure smallness of $a_{k}$ and we have
		\[
		\|a_k\|_{\mathrm{L}^{\infty}(B_1)}
		+
		\|f_k\|_{\mathrm{L}^{\infty}(B_1)}
		\le
		\delta.
		\]
		By \Cref{lemaC2}, applied to $v_k$, we obtain
		\[
		\sup_{B_{1/2}}v_k
		\le
		2^{-(2+\tilde\alpha)}.
		\]
		Scaling back gives
		\[
		\sup_{B_{2^{-(k+1)}}}u
		\le
		2^{-(k+1)(2+\tilde\alpha)}.
		\]
		This completes the induction and proves
		\eqref{eq:dyadic-schauder-estimate}.
		
		Now, let $0<r\le1/2$ and choose $k\ge1$ such that $2^{-(k+1)}<r\le2^{-k}$. Then,
		\[
		\sup_{B_r}u
		\le
		\sup_{B_{2^{-k}}}u
		\le
		2^{-k(2+\tilde\alpha)}.
		\]
		Since $2^{-(k+1)}<r$, we have $2^{-k(2+\tilde\alpha)}
		\le
		2^{2+\tilde\alpha}r^{2+\tilde\alpha}$. Thus, in the normalized regime,
		\begin{equation}\label{eq:normalized-schauder-final}
			\sup_{B_r}(u-u(0))
			\le
			C r^{2+\tilde\alpha},
			\quad\text{for }0<r\le1/2.
		\end{equation}
		
		\smallskip
		
		\noindent\textit{Step 2:} rescaling back.
		We now remove the normalization. Define, for $x\in B_1$,
		\[
		w(x):=
		\frac{u(\tau x)-u(0)}{K},
		\]
		where $K\ge1$ and $\tau\in(0,1)$ are to be chosen below. We have $w(0)=0$, and $0$ is a
		local minimum of $w$. Moreover, $w$ solves the equation of the same type,
		\[
		|\nabla w|^{\bar p(x)}\bar F(\nabla^2w)
		+
		\bar a(x)|\nabla w|^{\bar q(x)}
		=
		\bar f(x),
		\quad\text{in }B_1,
		\]
		where
		\[
		\bar p(x):=p(\tau x),
		\quad
		\bar q(x):=q(\tau x),
		\quad
		\bar F(M):=\frac{\tau^2}{K}F\left(\frac{K}{\tau^2}M\right),
		\]
		\[
		\bar f(x)
		=
		\frac{\tau^{\bar p(x)+2}}{K^{\bar p(x)+1}}f(\tau x),
		\quad
		\text{and}
		\quad
		\bar a(x)
		=
		\tau^{\bar p(x)-\bar q(x)+2}
		K^{\bar q(x)-\bar p(x)-1}
		a(\tau x).
		\]
		The operator $\bar F$ is uniformly $(\lambda,\Lambda)$-elliptic and
		$\bar F(0)=0$.
		Using \eqref{source1}, we get
		\[
		|\bar f(x)|
		\le
		K_1
		\frac{\tau^{\bar p(x)+2+\theta_1}}
		{K^{\bar p(x)+1}}
		|x|^{\theta_1}
		\le
		K_1\frac{\tau^{2+\theta_1}}{K}|x|^{\theta_1},
		\]
		and
		\[
		|\bar a(x)|
		\le
		K_2
		\tau^{\bar p(x)-\bar q(x)+2+\theta_2}
		K^{\bar q(x)-\bar p(x)-1}
		|x|^{\theta_2}
		\le
		K_2\frac{\tau^{2+\theta_2}}{K}|x|^{\theta_2}.
		\]
		Choose $K$ and $\tau$ so that $\|w\|_{\mathrm{L}^{\infty}(B_1)}\le1$,
		\[
		K_1\frac{\tau^{2+\theta_1}}{K}\le\frac{\delta}{2}
		\quad
		\text{and}
		\quad
		K_2\frac{\tau^{2+\theta_2}}{K}\le\frac{\delta}{2}.
		\]
		For instance, this can be achieved by taking $K$ large enough in terms of
		$\|u\|_{\mathrm{L}^{\infty}(B_1)}$, $K_1$, $K_2$ and then $\tau\in(0,1)$ fixed in terms of the
		universal parameters. Hence, $w$ satisfies the normalized assumptions from Step 1.
		
		Applying \eqref{eq:normalized-schauder-final} to $w$, we obtain
		\[
		\sup_{B_s}w
		\le
		Cs^{2+\tilde\alpha},
		\quad\text{for }0<s\le1/2.
		\]
		Rescaling back, with $r=\tau s$, gives
		\[
		\sup_{B_r}(u-u(0))
		\le
		CK\tau^{-(2+\tilde\alpha)}r^{2+\tilde\alpha},
		\quad\text{for }0<r\le\frac{\tau}{2}.
		\]
		Thus the theorem holds for $r_0:= \tau/2$, with a constant depending on $\|u\|_{\mathrm{L}^{\infty}(B_1)}$ and on the universal
		parameters, including $K_1,K_2,\theta_1,\theta_2$. Since $u(x)\ge u(0)$ near the
		origin, this is equivalent to
		\[
		\sup_{B_r}|u(x)-u(0)|
		\le
		C r^{2+\tilde\alpha},
		\quad\text{for }0<r\le r_0.
		\]
		The proof is complete.
	\end{proof}

	\medskip
	
	\noindent{\bf Acknowledgments.} G. Cosmo and R. R. Costa were partly supported by Ph.D. fellowships from the Conselho Nacional de Desenvolvimento Científico e Tecnológico (CNPq), Brazil. This study was financed in part by the Coordenação de Aperfeiçoamento de Pessoal de Nível Superior - Brasil (CAPES) - Finance Code 001.
	
	\medskip

	\bibliographystyle{amsplain, amsalpha}

\end{document}